\documentclass{article}

\usepackage{geometry}
\usepackage{graphicx} 
\usepackage{epstopdf} 

\usepackage[colorlinks, hypertexnames=false]{hyperref}
\hypersetup{linkcolor=blue, citecolor=red} 

\usepackage{bm}
\usepackage{amsmath} 
\usepackage{amssymb} 
\usepackage{stmaryrd} 
\usepackage{MnSymbol} 
\usepackage{multirow}

\usepackage{amsthm} 
\usepackage{mathtools} 

\usepackage[dvipsnames]{xcolor}
\definecolor{red}{rgb}{1.00,0.00,0.00}
{\numberwithin{equation}{section}
\setlength{\parindent}{1em}

\newtheorem{theorem}{Theorem}[section]
\newtheorem{lemma}{Lemma}[section]

\newcommand{\normmm}[1]{{\left\vert\kern-0.25ex\left\vert
\kern-0.25ex\left\vert #1
    \right\vert\kern-0.25ex\right\vert\kern-0.25ex\right\vert}}

\geometry{left=3cm,right=3cm,top=4cm,bottom=2.5cm}

\begin{document}           

\title{Adaptive staggered DG method for \\Darcy flows in fractured porous media}
\author{Lina Zhao\footnotemark[1]\qquad
    \;Eric Chung\footnotemark[4]}
\renewcommand{\thefootnote}{\fnsymbol{footnote}}
\footnotetext[1]{Department of Mathematics,The Chinese University of Hong Kong, Hong Kong SAR, China. ({lzhao@math.cuhk.edu.hk}).}
\footnotetext[2]{Department of Mathematics, The Chinese University of Hong Kong, Hong Kong SAR, China. ({tschung@math.cuhk.edu.hk}).}

\maketitle

\textbf{Abstract:}
Modeling flows in fractured porous media is important in applications. One main challenge in numerical simulation is that the flow is strongly influenced by the fractures, so that
the solutions typically contain complex features, which require high computational grid resolutions.
Instead of using uniformly fine mesh, a more computationally efficient adaptively refined mesh is desirable.
In this paper we design and analyze a novel residual-type a posteriori error estimator for staggered DG methods on general polygonal meshes for Darcy flows in fractured porous media. The method can handle fairly general meshes and hanging nodes can be simply incorporated into the construction of the method, which is highly appreciated for adaptive mesh refinement. The reliability and efficiency of the error estmator are proved. The derivation of the reliability hinges on the stability of the continuous setting in the primal formulation. A conforming counterpart that is continuous within each bulk domain for the discrete bulk pressure is defined to facilitate the derivation of the reliability. Finally, several numerical experiments including multiple non-intersecting fractures are carried out to confirm the proposed theories.


\pagestyle{myheadings} \thispagestyle{plain} \markboth{LinaEric}
    {Adaptive SDG method for Darcy flows in fractured porous media}

\section{Introduction}

Modeling flows in fractured porous media is of great importance thanks to its wide applications in many of the environmental and energy problems. In many of the applications the flow is strongly influenced by the presence of the fractures and it is challenging to effectively model the interaction between the system of fractures and the porous matrix. A popular choice for handling this problem is to treat fractures as $(d-1)$-dimensional interfaces between $d$-dimensional porous matrices, $d=2,3$. The development of this kind of reduced models has been addressed for single-phase Darcy flow \cite{Alboin02,Martin05,Frih08}, where the fracture flow equations and the proper interface conditions across the fractures are identified.

Numerous numerical methods have been developed for the approximation of the coupled bulk-fracture model, such as finite difference method, finite volume method, the Galerkin finite element method and mixed finite element method \cite{Monteagudo04,Hoteit08,DAngelo12}. Recently, polygonal methods have drawn great attention and several methods have been successfully applied to solve Darcy flows in fractured porous media, in this context we mention virtual element method, mimetic finite difference method, hybrid high-order method, discontinuous Galerkin method and staggered DG method \cite{Benedetto14,AntoniettiMF16,Chave18,Antonietti19,LinaDohyun20}. All these methods allow arbitrary shapes of polygon, which can greatly simplify the meshing process since the background grids can be generated independently of the fractures. Importantly, these methods allow hanging nodes, which is highly appreciated for adaptive mesh refinement.

Introduced in \cite{EricEngquistwave06,ChungWave}, staggered discontinuous Galerkin (DG) methods are new generation discretisation methods for PDEs based on discrete unknowns that enjoy staggered continuity properties. Inspired by the work given in \cite{EricEngquistwave06,ChungWave}, a large number of works have been dedicated to apply staggered DG methods to PDEs arising from practical applications \cite{Cheung15,ChungKimWid13,EricCiarYu13,KimChungLee13,LeeKim16,ChungQiu17,ChungLamQian15}. Recently, staggered DG method has been successfully design on fairly general polygonal meshes to solve Darcy law and the Stokes equations \cite{LinaPark,LinaParkShin}. It is further developed to solve the coupled Stokes and Darcy problem by properly enforcing the interface conditions \cite{LinaParkcoupledSD}. Another important contribution made by the authors is to relax the tangential continuity for velocity \cite{LinaEricLam20} so that the modified method is uniformly stable for Brinkman problem. Staggered DG methods designed therein earn many salient features, including:
1) It can be flexibly applied to general polygonal meshes with automatic treatment of the hanging nodes;
2) superconvergence can be obtained with suitable projection operator;
3) local mass conservations can be preserved, which is highly appreciated for the practical applications especially for the simulation of multiphase flow.
In addition, the mass matrix is block diagonal which is desirable when explicit time stepping schemes are used;
4) no numerical flux or penalty term is needed in contrast to other DG methods.
It is worth mentioning that staggered DG methods allow arbitrarily small edges \cite{LinaDohyun20}, which is important for the simulation of practical problems that encounter very irregular geometries such as cracking and gluing.
A relation of staggered DG method and the hybridized DG method is presented in \cite{ChungCockburnFu14,ChungCockburnFu16}.

Adaptive mesh refinement is an efficient procedure which can improve the quality of numerical approximations with minimal effort in particular for problems encounter singularities. The adaptive finite element method based on a posteriori error estimator is pioneered by Babu\v{s}ka and Rheinboldt \cite{BR78,Babuska78}. Since then a large number of works have been dedicated to a posteriori error estimators for second order elliptic problems, see e.g., \cite{Alonso95,Verfurth96,Bernadi00,Bress08,Braess96,Carsten97,
kim07,martin07,kipa-upwind08,LarsonAxel08,kimpark-sinum10,Vohralik10,ckp11,Ern15,ChungParkLina18}.
All the aforementioned error estimators are designed on triangular meshes, and the derivation of a posteriori error estimator for polygonal meshes is non-trivial and only a few works are available, in this direction one can refer to \cite{Beir08,Beir15,Cangiani16,Berrone17}. Deriving a posteriori error estimator for coupled bulk-fracture model is still in its infancy and only two works based on triangular meshes are available \cite{chen2016,ChenSun17}. To the best of our knowledge, no a posteriori error estimates for coupled bulk-fracture model based on general polygonal meshes have been studied in the literature so far. Therefore, the objective of this paper is to derive and analyze a novel residual-type a posteriori error estimator for staggered DG method proposed in \cite{LinaDohyun20} for Darcy flows in fractured porous media on general polygonal meshes.

In the formulation proposed in \cite{LinaDohyun20}, staggered DG method is used for the bulk domain and continuous finite element method is employed for the fracture model. The method can handle very general polygonal meshes and allows hanging nodes, which makes it desirable for adaptive mesh refinement. The derivation of a posteriori error estimator for this coupled model is non-trivial due to the coupling conditions imposed on the interface. To attack this issue, we decompose the discretization error into conforming part and nonconforming part via defining a conforming counterpart for the discrete bulk pressure. The nonconforming component can be estimated by using standard estimates and the upper bound for the conforming part hinges on the stability estimates derived for the continuous model in the primal formulation. Note that the conforming counterpart of the discrete solution is defined to be continuous within each bulk domain and no continuity is imposed for the fracture region. We can derive the upper bound for conforming part by using the residual equations obtained from the discretization error by naturally linking the discrete formulation and continuous formulation. The standard bubble functions are exploited to prove the efficiency of the proposed error estimator. Finally, several numerical experiments are tested, where we also include one example with multiple non-intersecting fractures. We find out that optimal convergence rates can be recovered by using adaptive mesh refinement guided by our error estimator. It is worth mentioning that our approach is different from the one used in \cite{ChenSun17}, where the partial continuous inf-sup condition is used. To the best of our knowledge, this is the first work on a posteriori error estimator for fractured porous media on polygonal meshes. We emphasize that our proof is quite general, thus it can be easily adjusted for other discretizations as well.

The rest of the paper is organized as follows. In the next section, we briefly introduce the model problem and describe the numerical scheme. Then in Section~\ref{sec:estimator}, the reliability and efficiency of the proposed error estimator are proved. Several numerical experiments are carried out in Section~\ref{sec:numerical} to verify the proposed theories. Finally, a conclusion is given.

%
%
%

\section{Description of staggered DG method}

In this section we first describe the model problem considered in this paper, then staggered DG discretization for the model problem is provided. We end by showing the stability of the continuous formulation by making use of the primal formulation.
\subsection{Model problem}
We consider a porous medium saturated by an incompressible fluid that occupies the space region $\Omega\subset \mathbb{R}^2$ and is crossed by a single fracture $\Gamma$. We focus our analysis on single fracture to avoid technical difficulties and the extension to multiple non-intersecting fractures is verified in our numerical simulation.
Here, $\Omega_B:=\Omega\backslash \bar{\Gamma}$ represents the bulk region and can be decomposed as $\Omega_B:=\Omega_{B,1}\cup \Omega_{B,2}$.
In addition, we denote by $\partial \Omega_B:=\bigcup_{i=1}^2 \partial \Omega_{B,i}\backslash \bar{\Gamma}$ and denote by $\partial \Gamma$ the boundary of fracture $\Gamma$.
$\bm{n}_\Gamma$ denotes a unit normal vector to $\Gamma$ with a fixed orientation.
The schematic of the bulk and fracture domain is illustrated in Figure~\ref{fig:bulkdomain}.
Without loss of generality, we assume in the following that the subdomains are numbered so that $\bm{n}_\Gamma$ coincides with the outward normal direction of $\Omega_{B,1}$.

In the bulk region, we model the motion of the incompressible fluid by Darcy's law in mixed form, so that the pressure $p: \Omega_B\rightarrow \mathbb{R}$ and the flux $\bm{u}: \Omega_B \rightarrow \mathbb{R}^2$ satisfy
\begin{align}
    \bm{u}+K\nabla p    & =\bm{0}\quad \mbox{in}\;\Omega_B,\label{eq:bulk1} \\
    \nabla \cdot \bm{u} & =f\quad \mbox{in}\;\Omega_B,\label{eq:bulk2} \\
    p                   & =p_0\quad \mbox{on}\; \partial \Omega_B.
\end{align}
Here, $p_0\in H^{\frac{1}{2}}(\partial \Omega_B)$ the boundary pressure, and $K:\Omega_B \rightarrow \mathbb{R}^{2\times 2}$ the bulk permeability tensor, which is assumed to be a symmetric, piecewise constant.
For the sake of simplicity we assume that $K$ is isotropic and positive definite.

Inside the fracture, we consider the motion of the fluid as governed by Darcy's law in primal form, so that the fracture pressure $p_\Gamma: \Gamma\rightarrow \mathbb{R}$ satisfies
\begin{equation}
    \begin{aligned}
        -\nabla_t \cdot (K_\Gamma \nabla_t p_\Gamma)
         & =\ell_{\Gamma}f_\Gamma+[\bm{u}\cdot \bm{n}_\Gamma]
         &                                                    & \mbox{in}\; \Gamma,          \\
        p_\Gamma
         & = g_\Gamma
         &                                                    & \mbox{on}\; \partial \Gamma,
    \end{aligned}
    \label{eq:fracture}
\end{equation}
where $f_\Gamma\in L^2(\Gamma)$ and $K_\Gamma: = \kappa_\Gamma^*\ell_\Gamma$ with $\kappa_\Gamma^*: \Gamma\rightarrow \mathbb{R}$ and $\ell_\Gamma:\Gamma\rightarrow \mathbb{R}$ denoting the tangential permeability and thickness of the fracture, respectively.
The quantities $\kappa_\Gamma^*$ and $\ell_\Gamma$ are assumed to be piecewise constants.
Here, $\nabla_t\cdot$ and $\nabla_t$ denote the tangential divergence and gradient operators along $\Gamma$, respectively.
For the sake of simplicity, we assume $p_0=0$, $g_\Gamma=0$ in the analysis.

The above problems are coupled by the following interface conditions
\begin{equation}
    \begin{aligned}
        \eta_\Gamma \{\bm{u}\cdot\bm{n}_\Gamma\} & =[p]            &  & \mbox{on}\;\Gamma, \\
        \alpha_\Gamma[\bm{u}\cdot\bm{n}_\Gamma]  & =\{p\}-p_\Gamma &  & \mbox{on}\;\Gamma,
    \end{aligned}\label{eq:interface}
\end{equation}
where we set
\begin{equation*}
    \eta_\Gamma: =\frac{\ell_\Gamma}{\kappa_\Gamma^n},\quad \alpha_\Gamma:=\eta_\Gamma(\frac{\xi}{2}-\frac{1}{4}).
\end{equation*}
Here $\xi\in (\frac{1}{2},1]$ is a model parameter, and $\kappa_\Gamma^n: \Gamma\rightarrow \mathbb{R}$ represents the normal permeability of the fracture, which is assumed to be a piecewise constant.
We assume that there exists positive constants $\kappa_1^*,\kappa_2^*,\kappa_1^n,\kappa_2^n$ such that, almost everywhere on $\Gamma$,
\begin{equation*}
    \kappa_1^*\leq \kappa_\Gamma^*\leq \kappa_2^*,\quad \kappa_1^n\leq \kappa_\Gamma^n\leq \kappa_2^n.
\end{equation*}

\begin{figure}
    \centering
    \includegraphics[width=0.5\textwidth]{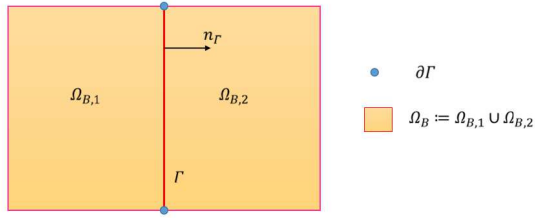}
    \caption{Illustration of bulk and fracture domain.}
    \label{fig:bulkdomain}
\end{figure}

To ease later analysis, we define $U=\{\psi\in H^1(\Omega_{B,1})\cap H^1(\Omega_{B,2}):\psi \mid_{\partial \Omega_B}=0\}$, $V_\Gamma=H^1_0(\Gamma)$, $V=U\times V_\Gamma$ and $Q=[L^2(\Omega_B)]^2$. Thereby we can propose the following weak formulation by employing integration by parts and the interface conditions \eqref{eq:interface}: Find $(\bm{u},p,p_\Gamma)\in Q\times U\times V_\Gamma$ such that
\begin{align}
        (K^{-1} \bm{u}, \bm{v})_{\Omega_B}+(\bm{v},\nabla p)_{\Omega_B}&=0\quad \forall \bm{v}\in Q,\label{eq:weak1}\\
       \sum_{e\in \mathcal{F}_h^\Gamma}\langle \frac{1}{\alpha_\Gamma}(\{p\}-p_{\Gamma}),\{q\}\rangle_e+\sum_{e\in \mathcal{F}_h^\Gamma}\langle \frac{1}{\eta_\Gamma}[p],[q]\rangle_e-(\bm{u},\nabla q)_{\Omega_B}&=(f,q)_{\Omega_B}\quad \forall q\in U,\label{eq:weak2}\\
        \langle K_\Gamma\nabla_t p_{\Gamma}, \nabla_t q_{\Gamma}\rangle_\Gamma-\sum_{e\in \mathcal{F}_h^\Gamma}\langle\frac{1}{\alpha_\Gamma}(\{p\}-p_{\Gamma}),q_{\Gamma}\rangle_e
        &=\langle\ell_\Gamma f_\Gamma, q_{\Gamma}\rangle_\Gamma\quad \forall q_\Gamma\in V_\Gamma.\label{eq:weak3}
\end{align}

Let
\begin{align*}
a(\bm{u},\bm{v}) &= (K^{-1}\bm{u},\bm{v})_{\Omega_B}, \quad b(\bm{v},p) = (\bm{v},\nabla p)_{\Omega_B},\quad c(p_\Gamma,q_\Gamma)=\langle K_\Gamma\nabla_t p_{\Gamma}, \nabla_t q_{\Gamma}\rangle_\Gamma,\\
I((p,p_\Gamma),(q,q_\Gamma)) & =\sum_{e\in \mathcal{F}_h^\Gamma}\langle \frac{1}{\alpha_\Gamma}(\{p\}-p_{\Gamma}),\{q\}-q_\Gamma\rangle_e+\sum_{e\in \mathcal{F}_h^\Gamma}\langle \frac{1}{\eta_\Gamma}[p],[q]\rangle_e
\end{align*}
and define
\begin{align*}
A((\bm{u},p,p_\Gamma),(\bm{v},q,q_\Gamma)): &= a(\bm{u},\bm{v})+b(\bm{v},p)
+I((p,p_{\Gamma}),(q,q_\Gamma))-b(\bm{u},q)+c(p_{\Gamma},q_\Gamma).
\end{align*}
Then \eqref{eq:weak1}-\eqref{eq:weak3} can be rewritten as: Find $(\bm{u},p,p_\Gamma)\in Q\times U\times V_\Gamma$ such that
\begin{align}
A((\bm{u},p,p_\Gamma),(\bm{v},q,q_\Gamma))=(f,q)_{\Omega_B}+\langle\ell_\Gamma f_\Gamma, q_{\Gamma}\rangle_\Gamma\quad \forall (\bm{v},q,q_\Gamma)\in Q\times U\times V_\Gamma.\label{eq:A}
\end{align}

Before closing this subsection, we introduce some notations that will be employed throughout the paper.
Let $D\subset \mathbb{R}^d,$ $d=1,2$, we adopt the standard notations for the Sobolev spaces $H^s(D)$ and their associated norms $\|\cdot\|_{s,D}$, and semi-norms $|\cdot|_{s,D}$ for $s\geq 0$.
The space $H^0(D)$ coincides with $L^2(D)$, for which the norm is denoted as $\|\cdot\|_{D}$.
We use $(\cdot,\cdot)_D$ to denote the inner product for $d=2$ and $\langle\cdot,\cdot\rangle_D$ for $d=1$.
In the sequel, we use $C$ to denote a generic positive constant which may have different values at different occurrences.

\subsection{Staggered DG method}

In this subsection, we begin with introducing the construction of our staggered DG spaces, in line with this we then present the staggered DG method for the model problem \eqref{eq:bulk1}-\eqref{eq:interface}.
We consider a family of meshes $\mathcal{T}_u$ made of disjoint polygonal (primal) elements which are aligned with the fracture $\Gamma$ so that any element $T\in \mathcal{T}_u$ can not be cut by $\Gamma$. We remark that our method can be easily adjusted for unfitted background grids, but we focus on the fitted case for simplicity.
Note that, since $\Omega_{B,1}$ and $\Omega_{B,2}$ are disjoint, each element $T$ belongs to one of the two subdomains.
The union of all the edges excluding the edges lying on the fracture $\Gamma$ in the decomposition $\mathcal{T}_u$ is called primal edges, which is denoted as $\mathcal{F}_u$.
Here we use $\mathcal{F}_u^0$ to stand for the subset of $\mathcal{F}_u$, that is the set of edges in $\mathcal{F}_{u}$ that do not lie on $\partial\Omega_B$.
In addition, we use $\mathcal{F}_h^\Gamma$ to denote the one-dimensional mesh of the fracture $\Gamma$.
For the construction of staggered DG method, we decompose each element $T \in \mathcal{T}_u$ into the union of triangles by connecting the interior point $\nu$ of $T$ to all the vertices.
Here the interior point $\nu$ is chosen as the center point for simplicity.
We rename the union of these sub-triangles by $S(\nu)$ to indicate that the triangles sharing common vertex $\nu$.
In addition, the resulting simplicial sub-meshes are denoted as $\mathcal{T}_h$.
Moreover, some additional edges are generated in the subdivision process due to the connection of $\nu$ to all the vertices of the primal element, and these edges are denoted by  $\mathcal{F}_p$.
For each triangle $\tau\in \mathcal{T}_h$, we let $h_\tau$ be the diameter of $\tau$ and $h=\max\{h_\tau, \tau\in \mathcal{T}_h\}$.
In addition, we define $\mathcal{F}:=\mathcal{F}_{u}\cup \mathcal{F}_{p}$ and $\mathcal{F}^{0}:=\mathcal{F}_{u}^{0}\cup \mathcal{F}_{p}$.
The construction for general meshes is illustrated in Figure~\ref{grid}, where the black solid lines are edges in $\mathcal{F}_{u}$ and black dotted lines are edges in $\mathcal{F}_{p}$.

Finally, we construct the dual mesh.
For each interior edge $e\in \mathcal{F}_{u}^0$, we use $D(e)$ to represent the dual mesh, which is the union of the two triangles in $\mathcal{T}_h$ sharing the edge $e$.
For each edge $e\in(\mathcal{F}_{u}\backslash\mathcal{F}_{u}^0)\cup \mathcal{F}_h^\Gamma$, we use $D(e)$ to denote the triangle in $\mathcal{T}_h$ having the edge $e$, see Figure~\ref{grid}.

For each edge $e$, we define a unit normal vector $\bm{n}_{e}$ as follows:
If $e\in \mathcal{F}\backslash \mathcal{F}^{0}$, then $\bm{n}_{e}$ is the unit normal vector of $e$ pointing towards the outside of $\Omega$.
If $e\in \mathcal{F}^{0}$, an interior edge, we then fix $\bm{n}_{e}$ as one of the two possible unit normal vectors on $e$.
When there is no ambiguity, we use $\bm{n}$ instead of $\bm{n}_{e}$ to simplify the notation.

We assume that our initial partition $\mathcal{T}_u$ satisfies the following mesh regularity assumptions (cf. \cite{Beir13,Cangiani16}):
\begin{description}
    \item[Assumption (A)] Every element $S(\nu)$ in $\mathcal{T}_{u}$ is star-shaped with respect to a ball of radius $\geq \rho_S h_{S(\nu)}$, where $\rho_S$ is a positive constant and $h_{S(\nu)}$ denotes the diameter of $S(\nu)$.
    \item[Assumption (B)] For every element $S(\nu)\in \mathcal{T}_{u}$ and every edge $e\in \partial S(\nu)$, it satisfies $h_e\geq \rho_E h_{S(\nu)}$, where $\rho_E$ is a positive constant and $h_e$ denotes the length of edge $e$.
\end{description}
We remark that Assumption (A) and (B) can guarantee that the triangulation $\mathcal{T}_h$ is shape regular.

\begin{figure}
    \centering
    \includegraphics[width=0.35\textwidth]{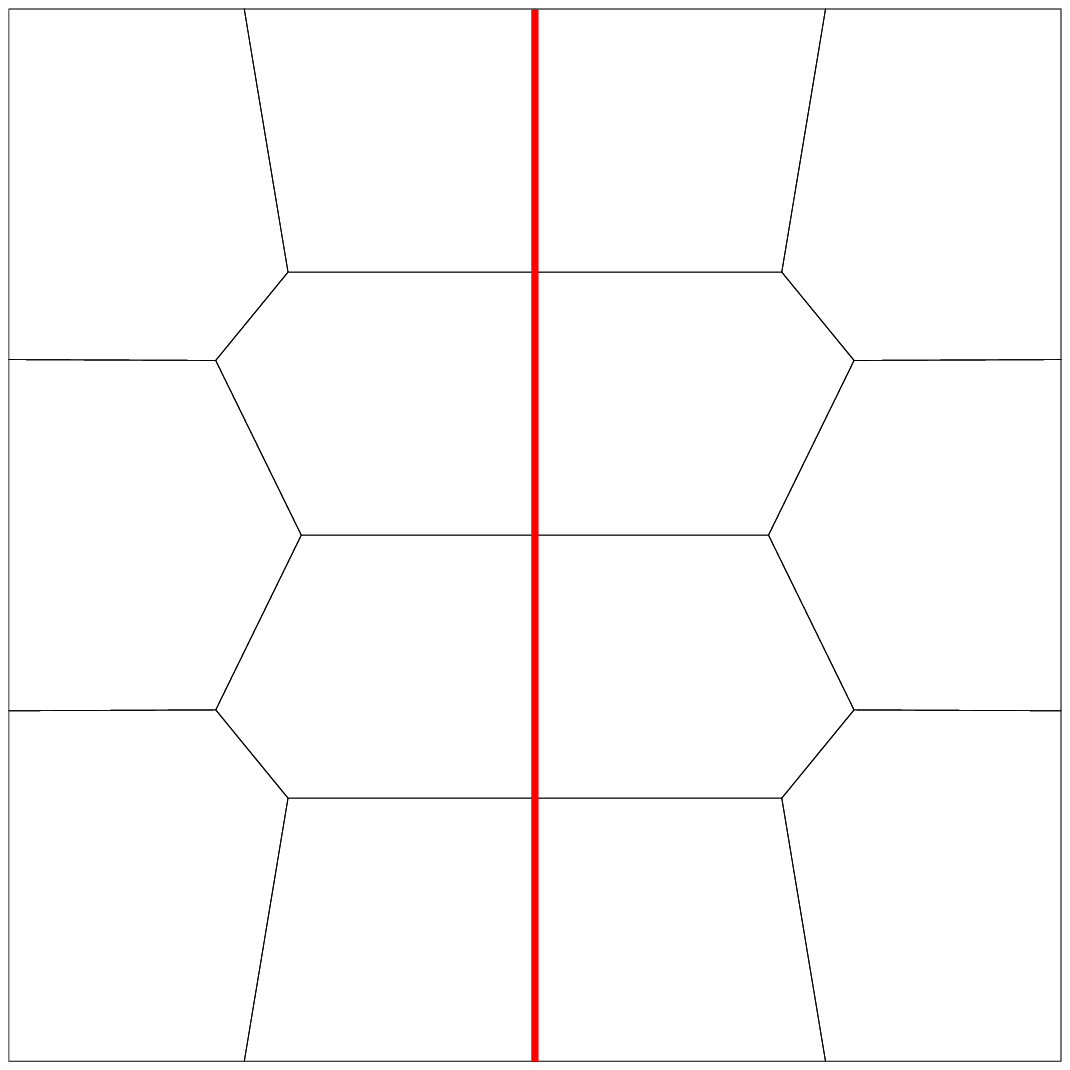}\hspace{1em}
    \includegraphics[width=0.35\textwidth]{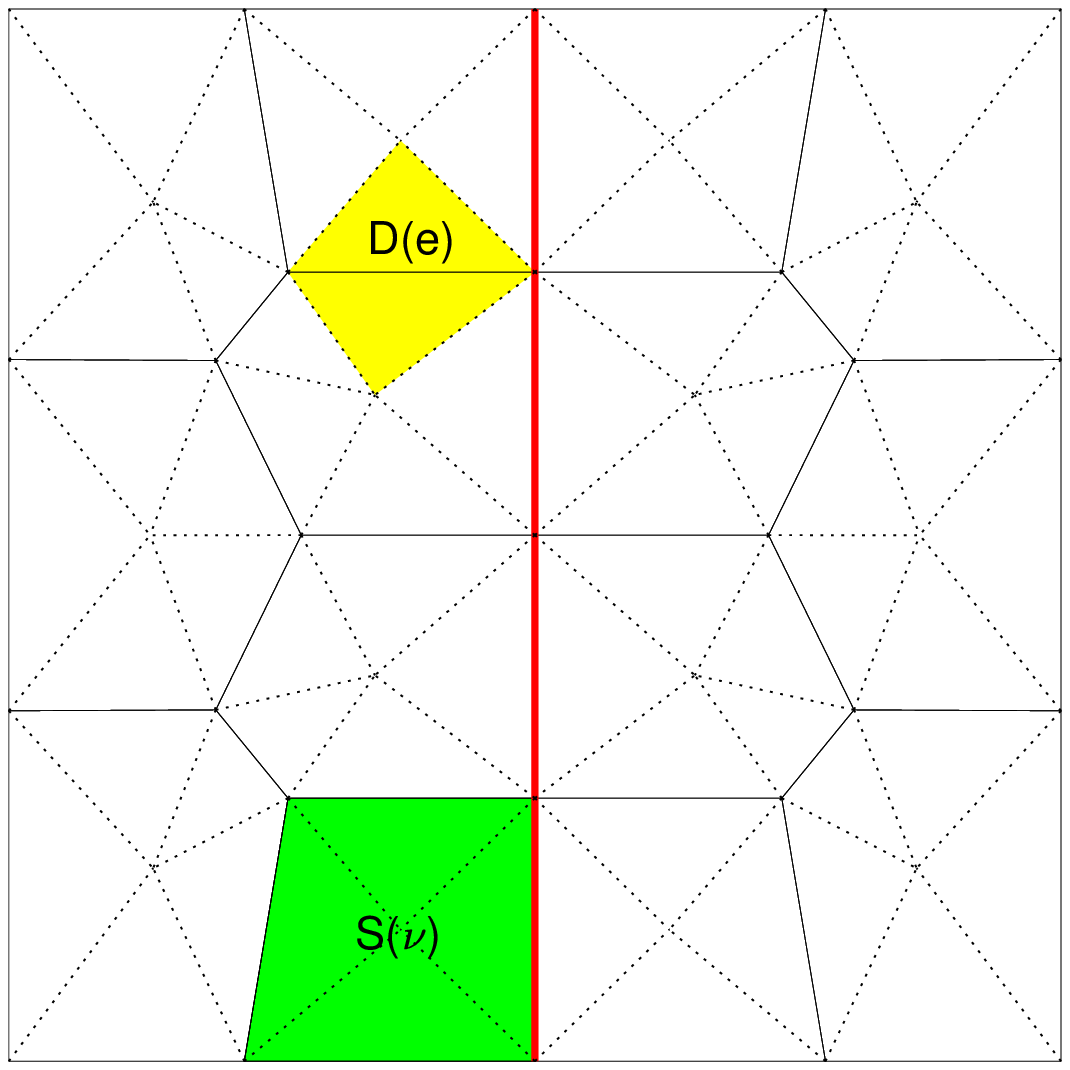}
    \caption{Schematic of the primal mesh $S(\nu)$, the dual mesh $D(e)$ and the primal simplicial sub-meshes.}
    \label{grid}
\end{figure}


Let $k\geq 0$ be the order of approximation. For every $\tau \in \mathcal{T}_{h}$ and $e\in\mathcal{F}$, we define $P^{k}(\tau)$ and $P^{k}(e)$ as the spaces of polynomials of degree less than or equal to $k$ on $\tau$ and $e$, respectively.
For $q$ and $\bm{v}$ belonging to the broken Sobolev space the jump $[q]\mid_e$ and the jump $[\bm{v}\cdot\bm{n}]\mid_e$ over $e\in \mathcal{F}^0\cup \mathcal{F}_h^\Gamma$ are defined respectively as
\begin{equation*}
    [q]=q_{1}-q_{2}, \quad [\bm{v}\cdot\bm{n}]=\bm{v}_{1}\cdot\bm{n}-\bm{v}_{2}\cdot\bm{n},
\end{equation*}
where $q_{i}=q\mid_{\tau_{i}}$, $\bm{v}_{i}=\bm{v}\mid_{\tau_{i}}$ and $\tau_{1}$, $\tau_{2}$ are the two triangles in $\mathcal{T}_h$ having the edge $e$.
Moreover, for $e\in \mathcal{F}\backslash \mathcal{F}^0$, we define $[q]=q_1$.
In the above definitions, we assume $\bm{n}$ is pointing from $\tau_1$ to $\tau_2$.

Similarly, we define the average $\{q\}\mid_e$ and the average $\{\bm{v}\cdot\bm{n}\}\mid_e$ over $e\in \mathcal{F}^0\cup \mathcal{F}_h^\Gamma$ by
\begin{align*}
    \{q\}=\frac{q_{1}+q_{2}}{2}, \quad \{\bm{v}\cdot\bm{n}\}=\frac{\bm{v}_{1}\cdot\bm{n}+\bm{v}_{2}\cdot\bm{n}}{2},
\end{align*}
where $q_{i}=q\mid_{\tau_{i}}$, $\bm{v}_{i}=\bm{v}\mid_{\tau_{i}}$ and $\tau_{1}$, $\tau_{2}$ are the two triangles in $\mathcal{T}_h$ having the edge $e$.

Next, we will introduce some finite dimensional spaces.
First, we define the following locally $H^{1}(\Omega)$ conforming space $S_h$:
\begin{equation*}
    S_{h}:=\{q :
    q\mid_{\tau}\in P^{k}(\tau)\;
    \forall \tau \in \mathcal{T}_{h};\;[q]\mid_e=0\;\forall e\in\mathcal{F}_{u}^{0};
    \;q\mid_{\partial \Omega_B}=0\}.
\end{equation*}
Notice that, if $q\in S_h$, then $q\mid_{D(e)}\in H^1(D(e))$ for each edge $e\in (\mathcal{F}_{u}\cup \mathcal{F}_h^\Gamma)$ and no continuity is imposed across $e\in \mathcal{F}_h^\Gamma$ for function $q\in S_h$.

We next define the following locally $H(\mbox{div};\Omega)-$conforming space $\bm{V}_h$:
\begin{equation*}
    \bm{V}_{h}=\{\bm{v}:
    \bm{v}\mid_{\tau} \in P^{k}(\tau)^{2}\;\forall \tau \in \mathcal{T}_{h};\;
    [\bm{v}\cdot\bm{n}]\mid_e=0\;\forall e\in \mathcal{F}_{p}\}.
\end{equation*}
Note that if $\bm{v}\in \bm{V}_h$, then $\bm{v}\mid_{S(\nu)}\in H(\textnormal{div};S(\nu))$ for each $S(\nu)\in \mathcal{T}_{u}$.
%
%

Finally, we define a finite dimensional subspace of $H^1_0(\Gamma)$ by
\begin{equation*}
    W_h=\{q_\Gamma: q_\Gamma\in H^1_0(\Gamma)\;|\; q_\Gamma\mid_e \in P^k(e), \forall e\in \mathcal{F}_h^\Gamma\}.
\end{equation*}

Then following \cite{LinaDohyun20}, we can achieve the discrete formulation for the model problem \eqref{eq:bulk1}-\eqref{eq:interface}: Find $(\bm{u}_h, p_h, p_{\Gamma,h})\in \bm{V}_h\times S_h\times W_h$ such that
\begin{equation}
    \begin{split}
        (K^{-1} \bm{u}_h, \bm{v})_{\Omega_B}+b_h^*(p_h, \bm{v})&=0\quad \forall \bm{v}\in \bm{V}_h,\\
        -b_h(\bm{u}_h, q)+\sum_{e\in \mathcal{F}_h^\Gamma}\langle \frac{1}{\alpha_\Gamma}(\{p_h\}-p_{\Gamma,h}),\{q\}\rangle_e+\sum_{e\in \mathcal{F}_h^\Gamma}\langle \frac{1}{\eta_\Gamma}[p_h],[q]\rangle_e&=(f,q)_{\Omega_B}\quad \forall q\in S_h,\\
        \langle K_\Gamma\nabla_t p_{\Gamma,h}, \nabla_t q_{\Gamma}\rangle_\Gamma-\sum_{e\in \mathcal{F}_h^\Gamma}\langle\frac{1}{\alpha_\Gamma}(\{p_h\}-p_{\Gamma,h}),q_{\Gamma}\rangle_e
        &=\langle\ell_\Gamma f_\Gamma, q_{\Gamma}\rangle_\Gamma\quad \forall q_\Gamma\in W_h,
    \end{split}
    \label{eq:discrete}
\end{equation}
where the bilinear forms are defined by
\begin{align*}
    b_h(\bm{u}_h, q)   & =- \sum_{e\in \mathcal{F}_p}\langle \bm{u}_h\cdot\bm{n},[q]\rangle_e+\sum_{\tau\in \mathcal{T}_h}(\bm{u}_h,\nabla q)_\tau,                                                                                 \\
    b_h^*(p_h, \bm{v}) & =\sum_{e\in \mathcal{F}_u^0}\langle p_h,[\bm{v}\cdot\bm{n}]\rangle_e-\sum_{\tau\in \mathcal{T}_h}(p_h,\nabla \cdot\bm{v})_\tau+\sum_{e\in \mathcal{F}_h^\Gamma}\langle[p_h],\{\bm{v}\cdot\bm{n}\}\rangle_e
    +\sum_{e\in \mathcal{F}_h^\Gamma}\langle\{p_h\},[\bm{v}\cdot\bm{n}]\rangle_e.
\end{align*}
To facilitate later analysis, we define the following norms for $(q_,q_\Gamma)\in V+(S_h\times W_h)$ and $\bm{v}\in Q+\bm{V}_h$
\begin{align*}
\|(q,q_\Gamma)\|_V^2&=\sum_{e\in\mathcal{F}_h^\Gamma}\|\alpha_\Gamma^{-1/2}(\{q\}-q_{\Gamma})\|_{0,e}^2
+\sum_{e\in\mathcal{F}_h^\Gamma}\|\eta_\Gamma^{-1/2}[q]\|_{0,e}^2+\|K^{1/2}\nabla q\|_{0,\Omega_B}^2+\| K_\Gamma^{1/2}\nabla_t q_{\Gamma}\|_{0,\Gamma}^2,\\
\|\bm{v}\|_{Q}&=\|K^{-1/2}\bm{v}\|_{0,\Omega_B},\quad \|(\bm{v},q, q_\Gamma)\|_{sdg}^2=\|\bm{v}\|_Q^2+\|(q,q_\Gamma)\|_V^2
+\|[\bm{v}\cdot\bm{n}_\Gamma]\|_{0,\Gamma}^2+\|\{\bm{v}\cdot\bm{n}_\Gamma\}\|_{0,\Gamma}^2.
\end{align*}

We state the following stability estimate, which is crucial for the subsequent analysis.
\begin{lemma}\label{lemma:dual-continuous}

For any $\ell_1\in Q^*$ and $\ell_2\in V^*$, where $Q^*$ and $V^*$ represent the dual spaces of $Q$ and $V$, respectively. Assume that $(\bm{u},p,p_\Gamma)\in Q\times U\times V_\Gamma$ satisfies
\begin{align}
        (K^{-1} \bm{u}, \bm{v})_{\Omega_B}+(\bm{v},\nabla p)_{\Omega_B}&=\ell_1(\bm{v}),\label{eq:weak1v}\\
       \langle K_\Gamma\nabla_t p_{\Gamma}, \nabla_t q_{\Gamma}\rangle_\Gamma+\sum_{e\in \mathcal{F}_h^\Gamma}\langle \frac{1}{\alpha_\Gamma}(\{p\}-p_{\Gamma}),\{q\}-q_\Gamma\rangle_e+\sum_{e\in \mathcal{F}_h^\Gamma}\langle \frac{1}{\eta_\Gamma}[p],[q]\rangle_e-(\bm{u},\nabla q)_{\Omega_B}&=\ell_2(q,q_\Gamma)\label{eq:weak2v},\\
\forall (\bm{v}, q, q_\Gamma)\in Q\times U\times V_\Gamma.\nonumber
\end{align}
Then, there exists a unique solution to \eqref{eq:weak1v}-\eqref{eq:weak2v} satisfying the following stability bound
\begin{align*}
\|(p,p_\Gamma)\|_{V}&\leq C (\|\ell_1\|_{Q^*}+\|\ell_2\|_{V^*}),\\
\|[\bm{u}\cdot\bm{n}_\Gamma]\|_{0,\Gamma}+\|\{\bm{u}\cdot\bm{n}_\Gamma\}\|_{0,\Gamma}
+\|K^{-1/2}\bm{u}\|_{Q}&\leq C(\|\ell_1\|_{Q^*}+\|\ell_2\|_{V^*}).
\end{align*}

\end{lemma}

\begin{proof}

Taking $\bm{v}=K\nabla q$ in \eqref{eq:weak1v}, we can get
\begin{align*}
(\bm{u}, \nabla q)_{\Omega_B}+(K\nabla p,\nabla q)_{\Omega_B}=\ell_1(K\nabla q),
\end{align*}
which can be combined with \eqref{eq:weak2v} yielding
\begin{align*}
&\sum_{e\in\mathcal{F}_h^\Gamma}\langle\frac{1}{\alpha_\Gamma}(\{p\}-p_{\Gamma}),\{q\}-q_\Gamma\rangle_e
+\sum_{e\in\mathcal{F}_h^\Gamma}\langle\frac{1}{\eta_\Gamma}[p],[q]\rangle_e\\
&+(K\nabla p,\nabla q)_{\Omega_B}+\langle K_\Gamma\nabla_t p_{\Gamma}, \nabla_t q_{\Gamma}\rangle_\Gamma=\ell_1(K\nabla q)+\ell_2(q,q_\Gamma).
\end{align*}
It is easy to check that the above formulation is well-posed, therefore, we can conclude that there exists a unique solution of $(p,p_\Gamma)\in V$, which satisfies
\begin{align*}
\|(p,p_\Gamma)\|_{V}\leq C (\sup_{q\in U}\frac{\ell_1(K\nabla q)}{\|K^{1/2}\nabla q\|_0}+\|\ell_2\|_{V^*}).
\end{align*}
Note that we have
\begin{align*}
\sup_{q\in U}\frac{\ell_1(K\nabla q)}{\|K^{1/2}\nabla q\|_{0,{\Omega_B}}}=\sup_{\bm{v}\in Q}\frac{\ell_1(\bm{v})}{\|K^{-1/2}\bm{v}\|_{0,{\Omega_B}}}\leq C\|\ell_1\|_{Q^*},
\end{align*}
thereby we can obtain
\begin{align*}
\|(p,p_\Gamma)\|_{V}\leq C(\|\ell_1\|_{Q^*}+\|\ell_2\|_{V^*}).
\end{align*}
On the other hand, we have by taking $\bm{v}=\bm{u}$ in \eqref{eq:weak1v}
\begin{align*}
(K^{-1}\bm{u},\bm{u})_{\Omega_B}=-(\bm{u},\nabla p)_{\Omega_B}+\ell_1(\bm{u}).
\end{align*}
Thus
\begin{align*}
\|K^{-1/2}\bm{u}\|_{0,\Omega_B}\leq \|K^{1/2}\nabla p\|_{0,\Omega_B}+\|\ell_1\|_{Q^*}\leq C(\|\ell_1\|_{Q^*}+\|\ell_2\|_{V^*}).
\end{align*}
Finally, we have $\alpha_\Gamma[\bm{u}\cdot\bm{n}_\Gamma]   =\{p\}-p_\Gamma$ and $\eta_\Gamma\{\bm{u}\cdot\bm{n}_\Gamma\}=[p]$ on $\Gamma$, hence
\begin{align*}
\|[\bm{u}\cdot\bm{n}_\Gamma]\|_{0,\Gamma}&\leq C \Big(\sum_{e\in \mathcal{F}_h^\Gamma}\|\alpha_\Gamma^{-1}(\{p\}-p_\Gamma)\|_{0,e}^2\Big)^{1/2}\leq C(\|\ell_1\|_{Q^*}+\|\ell_2\|_{V^*}),\\
\|\{\bm{u}\cdot\bm{n}_\Gamma\}\|_{0,\Gamma}\leq &C\Big(\sum_{e\in \mathcal{F}_h^\Gamma}\|\eta_\Gamma^{-1}[p]\|_{0,e}^2\Big)^{1/2}\leq C(\|\ell_1\|_{Q^*}+\|\ell_2\|_{V^*}).
\end{align*}
Therefore, the proof is completed.

\end{proof}

\section{Residual based a posteriori error estimator}\label{sec:estimator}

In this section we derive the reliability and efficiency of a residual type a posteriori error estimator, where the crux for the reliability is to use the stability estimate derived in Lemma~\ref{lemma:dual-continuous}. To this end, we need to define conforming counterpart of $p_h$ to incorporate into Lemma~\ref{lemma:dual-continuous}. The efficiency can be proved by employing bubble functions.

\subsection{Reliability}

To facilitate later analysis, we construct $p_h^{conf}\in M_h$ which is defined to be the conforming counterpart of $p_h$ within each bulk domain, and we require that $p_h^{conf}\mid_e=p_h\mid_e$ for any $e\in \partial \Omega_{B,i}\cap \Gamma$, where $M_h=\{\varphi\in C^0(\Omega_{B,1})\cap C^0(\Omega_{B,2}): \varphi \mid_{\partial \Omega_B}=0,\varphi\mid_\tau\in P^{k}(\tau),\forall \tau\in \mathcal{T}_h\}$. Indeed, $p_h^{conf}$ takes the same values as $p_h$ over the fracture region and is discontinuous therein. Therefore, we can obtain the following estimates proceeding analogously to Theorem~2.2 of \cite{Karakashian03} and the proof is omitted for simplicity.
\begin{lemma}\label{lemma:nonconforming}

The following estimate holds
\begin{align*}
\|\nabla (p_h-p_h^{conf})\|_{0,\Omega_B}^2\leq C \sum_{e\in \mathcal{F}_p}h_e^{-1}\|[p_h]\|_{0,e}^2.
\end{align*}

\end{lemma}

\begin{lemma}\label{lemma:conf}
Let $(\bm{u},p,p_\Gamma)$ be the weak solution of \eqref{eq:weak1}-\eqref{eq:weak3} and let $(\bm{u}_h,p_h,p_{\Gamma,h})$ be the discrete solution of \eqref{eq:discrete}, then for $p_h^{conf}\in M_h$, we have
\begin{align*}
&\Big(\|[\bm{u}\cdot\bm{n}_\Gamma-\bm{u}_h\cdot\bm{n}_\Gamma]\|_{0,\Gamma}^2
+\|\{\bm{u}\cdot\bm{n}_\Gamma-\bm{u}_h\cdot\bm{n}_\Gamma\}\|_{0,\Gamma}^2+\|\bm{u}-\bm{u}_h\|_{Q}^2\\
&\;+\|(p-p_h^{conf},p_\Gamma-p_{\Gamma,h})\|_{V}^2\Big)^{1/2}\leq C(\|\ell_1\|_{Q^*}+\|\ell_2\|_{V^*}).
\end{align*}
where
\begin{align}
\ell_1(\bm{v}):&=-a(\bm{u}_h,\bm{v})-b(p_h^{conf},\bm{v})\quad \forall \bm{v}\in Q,\label{eq:l1}\\
\ell_2(q,q_\Gamma):&=(f,q)_{\Omega_B}+\langle\ell_\Gamma f_\Gamma,q_{\Gamma}\rangle_\Gamma
-I((p_h^{conf},p_{\Gamma,h}),(q,q_\Gamma))+b(q,\bm{u}_h)-c(p_{\Gamma,h},q_\Gamma)\; \forall (q_\Gamma,q_{\Gamma})\in V.\label{eq:l2}
\end{align}

\end{lemma}

\begin{proof}

%
We can infer from the definition of $A$ and \eqref{eq:A}
\begin{align*}
&A((\bm{u}-\bm{u}_h,p-p_h^{conf},p_\Gamma-p_{\Gamma,h}),(\bm{v},q,q_\Gamma))\\
&=A((\bm{u},p,p_\Gamma),(\bm{v},q,q_\Gamma))-A((\bm{u}_h,p_h^{conf},p_{\Gamma,h}),(\bm{v},q,q_\Gamma))\\
&=(f,q)_{\Omega_B}+\langle\ell_\Gamma f_\Gamma,q_{\Gamma}\rangle_\Gamma-A((\bm{u}_h,p_h^{conf},p_{\Gamma,h}),(\bm{v},q,q_\Gamma))\\
&=(f,q)_{\Omega_B}+\langle\ell_\Gamma f_\Gamma,q_{\Gamma}\rangle_\Gamma-\Big(a(\bm{u}_h,\bm{v})+b(\bm{v},p_h^{conf})
+I((p_h^{conf},p_{\Gamma,h}),(q,q_\Gamma))-b(\bm{u}_h,q)+c(p_{\Gamma,h},q_\Gamma)\Big),
\end{align*}
where $(\bm{v},q,q_\Gamma)\in Q\times U\times V_\Gamma$.

Notice that $\bm{u}-\bm{u}_h\in Q$ and $(p-p_h^{conf},p_\Gamma-p_{\Gamma,h}) \in V$. Therefore, we can infer from Lemma~\ref{lemma:dual-continuous} that
\begin{align*}
&\Big(\|[\bm{u}\cdot\bm{n}_\Gamma-\bm{u}_h\cdot\bm{n}_\Gamma]\|_{0,\Gamma}^2
+\|\{\bm{u}\cdot\bm{n}_\Gamma-\bm{u}_h\cdot\bm{n}_\Gamma\}\|_{0,\Gamma}^2+\|\bm{u}-\bm{u}_h\|_{Q}^2\\
&\;+\|(p-p_h^{conf},p_\Gamma-p_{\Gamma,h})\|_{V}^2\Big)^{1/2}\leq C(\|\ell_1\|_{Q^*}+\|\ell_2\|_{V^*}).
\end{align*}

\end{proof}

In the sequel we use $I_h:U\rightarrow M_h$ to represent the Scott-Zhang interpolation operator defined in 2D. More precisely, we define $I_h=\sum_{i=1}^2I_h^i,i=1,2$ by
\begin{align*}
I_h^i \phi = \sum_{z\in \mathcal{N}_h(\Omega_{B,i})}(\Pi_z\phi)\varphi_z,
\end{align*}
where $\mathcal{N}_h(\Omega_{B,i})$ is the set of degrees of freedom for $M_h$ restricted to $\Omega_{B,i}$, $\varphi_z$ is the local basis function on $z$, $\Pi_z\phi=\int_{\sigma_z}\phi \theta_z$. Here $\sigma_z$ is an adjustable edge or triangle, see \cite{ScottZhang90} for more details. $\theta_z$ is the dual basis function of $\varphi_z$ on $\sigma_z$. Let $ \mathcal{N}_h(\Gamma)$ be the set of degrees of freedom of $M_h$ restricted on $\mathcal{F}_h^\Gamma$. If $z\in \mathcal{N}_h(\Gamma)$, we choose $\sigma_z\in \mathcal{F}_h^\Gamma$. Note that no continuity is imposed for $I_h$ across the fracture region.
Analogously, we define $\pi_h: H^1_0(\Gamma)\rightarrow W_h$ to be the Scott-Zhang interpolation operator associated to the degrees of freedom in $\mathcal{F}_h^\Gamma$. The following estimates can be found in \cite{ScottZhang90}.
\begin{lemma}\label{lemma:SZ}
For any $\tau\in \mathcal{T}_h$ and $e\in \mathcal{F}^0\cup \mathcal{F}_h^\Gamma$, the following estimates hold
\begin{align*}
\|q-I_h^iq\|_{0,\tau}&\leq C h_\tau |q|_{1,\omega_\tau}\quad \forall q\in H^1(\omega_\tau),\\
\|q-I_h^iq\|_{0,e}&\leq C h_e^{1/2}|q|_{1,\omega_e}\quad \forall q\in H^1(\omega_e),
\end{align*}
where $\omega_\tau=\cup \{\tau'\in \mathcal{T}_h\cap \Omega_{B,i}: \bar{\tau}'\cap\bar{\tau}\neq \emptyset\}$ and $\omega_e=\cup \{\tau'\in \mathcal{T}_h\cap \Omega_{B,i}: \bar{\tau}'\cap\bar{e}\neq \emptyset\}$.

In addition, let $\mathcal{N}_h^0(\Gamma)$ denote the set of interior vertices on $\mathcal{F}_h^\Gamma$, we have for any $e\in \mathcal{F}_h^\Gamma$ and $z\in \mathcal{N}_h^0(\Gamma)$
\begin{align*}
|(q_\Gamma-\pi_hq_\Gamma)_z|&\leq C h_e^{1/2}|q_\Gamma|_{1,\mathcal{N}_z},\\
\|q_\Gamma-\pi_hq_\Gamma\|_{0,e}&\leq C h_e|q_\Gamma|_{1,\mathcal{N}_e},
\end{align*}
where $\mathcal{N}_z=\cup \{e'\in \mathcal{F}_h^\Gamma: \bar{e}'\cap z\neq \emptyset\}$ and $\mathcal{N}_e=\cup \{e'\in \mathcal{F}_h^\Gamma: \bar{e}'\cap\bar{e}\neq \emptyset\}$.

\end{lemma}

\begin{lemma}\label{lemma:dual}

The linear functionals defined in \eqref{eq:l1}-\eqref{eq:l2} satisfy the following bound
\begin{align*}
\|\ell_1\|_{Q^*}+\|\ell_2\|_{V^*}\leq C \eta,
\end{align*}
where
\begin{align*}
\eta&=\Big(\Big(\sum_{\tau\in \mathcal{T}_h}\|K^{-1/2}\bm{u}_h+K^{1/2}\nabla p_h\|_{0,\tau}^2\Big)^{1/2}+\Big(\sum_{\tau\in \mathcal{T}_h}h_\tau^2\|f-\nabla \cdot \bm{u}_h\|_{0,\tau}^2\Big)^{1/2}+\Big(\sum_{e\in \mathcal{F}_p}h_e^{-1}\|[p_h]\|_{0,e}^2\Big)^{1/2}\\
&\;+\Big(\sum_{e\in \mathcal{F}_u^0}h_e\|[\bm{u}_h\cdot\bm{n}]\|_{0,e}^2\Big)^{1/2}+\Big(\sum_{e\in \mathcal{F}_h^\Gamma}h_e^2\|\ell_\Gamma f_\Gamma+\nabla \cdot (K_\Gamma \nabla_t p_{\Gamma,h})+[\bm{u}_h\cdot\bm{n}_\Gamma]\|_{0,e}^2\Big)^{1/2}\\
&\;+\Big(\sum_{z\in \mathcal{N}_h^0(\Gamma)}h_z[K_\Gamma^{1/2}\nabla_tp_{\Gamma,h}]_z^2\Big)^{1/2}+\Big(\sum_{e\in \mathcal{F}_h^\Gamma}h_e
\|\frac{1}{\alpha_\Gamma}(\{p_h\}-p_{\Gamma,h})-[\bm{u}_h\cdot\bm{n}_\Gamma]\|_{0,e}^2\Big)^{1/2}\\
&\;+\Big( \sum_{e\in\mathcal{F}_h^\Gamma}h_e\|\{\bm{u}_h\cdot\bm{n}_\Gamma\}-\frac{1}{\eta_\Gamma}[p_h]\|_{0,e}^2\Big)^{1/2}\Big).
\end{align*}
Here $h_z=\max\{h_e: z\in \partial e, e\in \mathcal{F}_h^\Gamma\}$ and $[\nabla_tp_{\Gamma,h}]_z=\nabla_tp_{\Gamma,h}\mid_e(z)-\nabla_tp_{\Gamma,h}\mid_{e'}(z)$ for any $z\in \mathcal{N}_h^0(\Gamma)$ and $\partial e\cap \partial e'=z$.
\end{lemma}

\begin{proof}

First, we have from the definition of $\ell_1(\bm{v})$ (cf. \eqref{eq:l1})
\begin{align*}
\ell_1(\bm{v}):&=-a(\bm{u}_h,\bm{v})-b(\bm{v},p_h^{conf})=-(K^{-1}\bm{u}_h,\bm{v})_{\Omega_B}
-(\bm{v},\nabla p_h^{conf})_{\Omega_B}\\
&=-(K^{-1}\bm{u}_h+\nabla p_h,\bm{v})_{\Omega_B}+(\bm{v},\nabla (p_h-p_h^{conf}))_{\Omega_B}\\
&\leq C\Big( \|K^{-1/2}\bm{u}_h+K^{1/2}\nabla p_h\|_{0,\Omega_B}+\|K^{1/2}\nabla (p_h-p_h^{conf})\|_{0,\Omega_B}\Big)\|K^{-1/2}\bm{v}\|_{0,\Omega_B}.
\end{align*}
On the other hand, we can decompose $\ell_2(q,q_\Gamma)$ as
\begin{align*}
\ell_2(q,q_\Gamma)=\ell_2(I_hq,\pi_hq_\Gamma)+\ell_2(q-I_hq,q_\Gamma-\pi_hq_\Gamma).
\end{align*}
We can infer from \eqref{eq:discrete} that
\begin{align*}
&(f,I_hq)_{\Omega_B}-\sum_{e\in \mathcal{F}_h^\Gamma}\langle \frac{1}{\alpha_\Gamma}(\{p_h\}-p_{\Gamma,h}),\{I_hq\}\rangle_e-\sum_{e\in \mathcal{F}_h^\Gamma}\langle \frac{1}{\eta_\Gamma}[p_h],[I_hq]\rangle_e+(\bm{u}_h,\nabla (I_hq))_{\Omega_B}=0
\end{align*}
and
\begin{align*}
\langle\ell_\Gamma f_\Gamma,\pi_hq_{\Gamma}\rangle_\Gamma+\sum_{e\in \mathcal{F}_h^\Gamma}\langle \frac{1}{\alpha_\Gamma}(\{p_h\}-p_{\Gamma,h}),\pi_h q_\Gamma\rangle_e-\langle K_\Gamma\nabla_t p_{\Gamma,h}, \nabla_t (\pi_h q_\Gamma)\rangle_\Gamma=0.
\end{align*}
Thus
\begin{align*}
\ell_2(I_hq,\pi_hq_\Gamma)=(f,I_hq)_{\Omega_B}+\langle \ell_\Gamma f_\Gamma,q_\Gamma\rangle_\Gamma-I((p_h^{conf},p_{\Gamma,h}),(q,q_\Gamma))+b(I_hq,\bm{u}_h)
-c(p_{\Gamma,h},q_\Gamma)=0,
\end{align*}
where we use the fact that $p_h\mid_e=p_h^{conf}\mid_e\;\forall e\in \mathcal{F}_h^\Gamma$, thus $I((p_h^{conf},p_{\Gamma,h}),(q,q_\Gamma))=I((p_h,p_{\Gamma,h}),(q,q_\Gamma))$.
It remains to estimate $\ell_2(q-I_hq,q_\Gamma-\pi_hq_\Gamma)$. We have from \eqref{eq:l2} and integration by parts
\begin{align*}
\ell_2(q-I_hq,q_\Gamma-\pi_hq_\Gamma)&=(f-\nabla \cdot \bm{u}_h,q-I_hq)_{\Omega_B}+\sum_{e\in \mathcal{F}_u^0}\langle [\bm{u}_h\cdot\bm{n}],q-I_hq\rangle_e-\sum_{z\in \mathcal{N}_h^0} [K_\Gamma \nabla p_{\Gamma,h}\cdot\bm{t}] (q_\Gamma-\pi_hq_{\Gamma})\\
&\;+\langle \ell_\Gamma f_\Gamma+\nabla_t \cdot (K_\Gamma \nabla_t p_{\Gamma,h})+[\bm{u}_h\cdot\bm{n}_\Gamma],q_{\Gamma}-\pi_h q_\Gamma\rangle_\Gamma\\
&\;+\langle\frac{1}{\alpha_\Gamma}(\{p_h\}-p_{\Gamma,h})-[\bm{u}_h\cdot\bm{n}_\Gamma],
q_\Gamma-\pi_hq_\Gamma\rangle_e+
\sum_{e\in\mathcal{F}_h^\Gamma}\langle\{\bm{u}_h\cdot\bm{n}_\Gamma\}-\frac{1}{\eta_\Gamma}[p_h],
[q-I_hq]\rangle_e\\
&\;+\sum_{e\in \mathcal{F}_h^\Gamma}\langle [\bm{u}_h\cdot\bm{n}_\Gamma]-\frac{1}{\alpha_\Gamma}(\{p_h\}-p_{\Gamma,h}),\{q-I_hq\}\rangle_e:=\sum_{i=1}^7I_i.
\end{align*}
We can estimate $I_i,i=1,\cdots,7$ by Lemma~\ref{lemma:SZ}
\begin{align*}
I_1&\leq \Big(\sum_{\tau\in \mathcal{T}_h}h_\tau^2\|f-\nabla \cdot \bm{u}_h\|_{0,\tau}^2\Big)^{1/2}\|\nabla q\|_{0,\Omega_B},\\
I_2&\leq C \Big(\sum_{e\in \mathcal{F}_u^0}h_e\|[\bm{u}_h\cdot\bm{n}]\|_{0,e}^2\Big)^{1/2}\|\nabla q\|_{0,\Omega_B},\\
I_3&\leq C \Big(\sum_{z\in \mathcal{N}_h^0(\Gamma)}h_z[K_\Gamma^{1/2}\nabla_tp_{\Gamma,h}]_z^2\Big)^{1/2}\|K_\Gamma^{1/2}\nabla q_\Gamma\|_{0,\Gamma},\\
I_4&\leq C\Big(\sum_{e\in \mathcal{F}_h^\Gamma}h_e^2\|\ell_\Gamma f_\Gamma+\nabla_t \cdot (K_\Gamma\nabla_t p_{\Gamma,h})+[\bm{u}_h\cdot\bm{n}_\Gamma]\|_{0,e}^2\Big)^{1/2}\|\nabla q_\Gamma\|_{0,\Gamma},\\
I_5&\leq C\Big(\sum_{e\in \mathcal{F}_h^\Gamma}h_e^2
\|\frac{1}{\alpha_\Gamma}(\{p_h\}-p_{\Gamma,h})-[\bm{u}_h\cdot\bm{n}_\Gamma]\|_{0,e}^2\Big)^{1/2}\|\nabla q_\Gamma\|_{0,\Gamma},\\
I_6&\leq C \Big( \sum_{e\in\mathcal{F}_h^\Gamma}h_e\|\{\bm{u}_h\cdot\bm{n}_\Gamma\}-\frac{1}{\eta_\Gamma}[p_h]\|_{0,e}^2\Big)^{1/2}
\|\nabla q\|_{0,\Omega_B}\\
I_7&\leq C \Big(\sum_{e\in \mathcal{F}_h^\Gamma}h_e\| [\bm{u}_h\cdot\bm{n}]-\frac{1}{\alpha_\Gamma}(\{p_h\}-p_{\Gamma,h})\|_{0,e}^2\Big)^{1/2}\|\nabla q\|_{0,\Omega_B}.
\end{align*}
Therefore, the proof is completed by combining the preceding arguments and Lemma~\ref{lemma:nonconforming}.
\end{proof}

Then we can state the main result of this subsection.
\begin{theorem}
There exists a positive constant $C$ independent of the meshsize such that
\begin{align*}
\|(\bm{u}-\bm{u}_h,p-p_h, p_\Gamma-p_{\Gamma,h})\|_{sdg}\leq C \eta,
\end{align*}

\end{theorem}

\begin{proof}

Triangle inequality implies
\begin{align*}
\|(p-p_h, p_\Gamma-p_{\Gamma,h})\|_{V}\leq \|(p-p_h^{conf},p_\Gamma-p_{\Gamma,h})\|_V
+\|(p_h^{conf}-p_h,p_\Gamma-p_{\Gamma,h})\|_V.
\end{align*}
Then an application of Lemma~\ref{lemma:nonconforming}, Lemma~\ref{lemma:conf} and Lemma~\ref{lemma:dual} completes the proof.

\end{proof}


\subsection{Efficiency}
In this subsection, we present the efficiency estimates.
To this end, we set the element bubble function in $\tau$ as $\psi_\tau$ and edge bubble function for each edge $e$ as $\psi_e$, and the properties of the bubble functions are given in the next lemma (cf. \cite{Verfurth96}).
\begin{lemma}\label{bubblel}
The following inequalities hold for all functions $v\in P^k(\tau)$.
\begin{align}
&\|v\|_{0,\tau}\leq C\|\psi_{\tau}^{1/2}v\|_{0,\tau}\leq C \|v\|_{0,\tau},\label{bubble1}\\
&\|\nabla (\psi_{\tau} v)\|_{0,\tau}\leq C h_{\tau}^{-1} \|v\|_{0,\tau}\label{bubble2}.
\end{align}
For an edge $e\in \mathcal{F}_u^0\cup \mathcal{F}_h^\Gamma$, we keep the same notation for the extension of the function $v\in P^k(e)$, originally only defined on the edge $e$, to a function defined on $D(e)$. The extension is done by constant values in the direction of the barycenter of e–opposite vertex. Then, we have
\begin{equation}
h_e^{1/2}\|v\|_{0,e}\leq C \|\psi_e v\|_{0,D(e)}\leq C h_e^{1/2} \|v\|_{0,e}\label{extension}
\end{equation}
and
\begin{align}
\|v\|_{0,e}^2\leq C \langle v, \psi_ev\rangle_e.\label{eq:tr1}
\end{align}

\end{lemma}

\begin{lemma}\label{lemma:eff1}
Let $(\bm{u}_h,p_h,p_{\Gamma,h})$ be the discrete solution of \eqref{eq:discrete}, let $f_h$ be the piecewise polynomial approximation of $f$ and let $f_{\Gamma,h}$ be the piecewise polynomial approximation of $f_\Gamma$. Then we have
\begin{align}
h_\tau \|f-\nabla \cdot\bm{u}_h\|_{0,\tau}&\leq C \|
K^{-\frac{1}{2}}(\bm{u}-\bm{u}_h)\|_{0,\tau}+h_\tau \|f-f_h\|_{0,\tau}\quad\forall \tau\in \mathcal{T}_h,\label{eq:eff-f}\\
\|K^{-1/2}\bm{u}_h+K^{1/2}\nabla p_h\|_{0,\tau}&\leq C\|K^{-1/2}(\bm{u}_h-\bm{u})\|_{0,\tau}+\|K^{1/2}\nabla (p- p_h)\|_{0,\tau}\quad \forall \tau\in \mathcal{T}_h,\label{eq:effup} \\
h_e^{1/2}\|[\bm{u}_h\cdot\bm{n}]\|_{0,e}&\leq C\Big(\Big(\sum_{\tau\in D_e}\|K^{-\frac{1}{2}}(\bm{u}-\bm{u}_h)\|_{0,\tau}^2
\Big)^{\frac{1}{2}}+\Big(\sum_{\tau\in D_e}h_\tau^2
\|f-f_h\|_{0,\tau}^2\Big)^{\frac{1}{2}}\Big)\quad\forall e\in \mathcal{F}_u^0\label{eq:effujump}.
\end{align}
In addition, it also holds for $e\in \mathcal{F}_h^\Gamma$
\begin{equation}
\begin{split}
h_e\|\ell_\Gamma f_{\Gamma}+\nabla_t \cdot (K_\Gamma\nabla_t p_{\Gamma,h})+[\bm{u}_h\cdot\bm{n}_\Gamma]\|_{0,e}&\leq  C\Big(\|K_\Gamma \nabla_t (p_\Gamma-p_{\Gamma,h})\|_{0,e}+h_e\|[\bm{u}\cdot\bm{n}_\Gamma-\bm{u}_h\cdot\bm{n}_\Gamma]\|_{0,e}\\
&\;+h_e\|\ell_\Gamma (f_{\Gamma,h}-f_\Gamma)\|_{0,e}\Big).
\end{split}
\label{eq:eff-fracture}
\end{equation}

\end{lemma}

\begin{proof}

Let $R_\tau(f_h):=f_h-\nabla \cdot\bm{u}_h$, then Green's theorem, the Cauchy-Schwarz inequality, \eqref{bubble1} and \eqref{bubble2} imply
\begin{align*}
(R_\tau(f_h), \psi_\tau R_\tau(f_h))_{\tau}&=(\nabla \cdot (\bm{u}-\bm{u}_h), \psi_\tau R_\tau(f_h))_\tau+(f_h-f, \psi_\tau R_\tau(f_h))_\tau\\
&=(K^{-\frac{1}{2}}(\bm{u}-\bm{u}_h),
K^{\frac{1}{2}}\nabla( \psi_\tau R_\tau(f_h)))_\tau+(f_h-f, \psi_\tau R_\tau(f_h))_\tau\\
&\leq C\Big(\|K^{-\frac{1}{2}}(\bm{u}-\bm{u}_h)\|_{0,\tau}
\|K^{\frac{1}{2}}\nabla (\psi_\tau R_\tau(f_h))\|_{0,\tau}+\|f-f_h\|_{0,\tau}\|\psi_\tau R_\tau(f_h)\|_{0,\tau}\Big)\\
&\leq C\Big( \|K^{-\frac{1}{2}}(\bm{u}-\bm{u}_h)\|_{0,\tau}
h_{\tau}^{-1}K^{\frac{1}{2}}\|R_\tau(f_h)\|_{0,\tau}
+\|f-f_h\|_{0,\tau}\| R_\tau(f_h)\|_{0,\tau}\Big),
\end{align*}
Combining the above inequality with inequality \eqref{bubble1}, we can achieve
\begin{eqnarray*}
\|R_\tau(f_h)\|_{0,\tau}^2\leq C\Big( \|K^{-\frac{1}{2}}(\bm{u}-\bm{u}_h)\|_{0,\tau}h_\tau^{-1} K^{\frac{1}{2}}\|R_\tau(f_h)\|_{0,\tau}+\|f-f_h\|_{0,\tau} \|R_\tau(f_h)\|_{0,\tau}\Big),
\end{eqnarray*}
which gives \eqref{eq:eff-f}.

Then, we can estimate \eqref{eq:effup} by triangle inequality and the relationship $\bm{u}=-K\nabla p$
\begin{align*}
\|K^{-1/2}\bm{u}_h+K^{1/2}\nabla p_h\|_{0,\tau}\leq \|K^{-1/2}(\bm{u}_h-\bm{u})\|_{0,\tau}+\|K^{1/2}\nabla (p- p_h)\|_{0,\tau}.
\end{align*}
Next, we estimate \eqref{eq:effujump}. Green's theorem yields
\begin{align*}
([\bm{u}_h]\cdot \bm{n}, \psi_e[\bm{u}_h\cdot \bm{n}])_e&=\sum_{\tau\in D_e}((\bm{u}_h-\bm{u})\cdot \bm{n},\psi_e[\bm{u}_h\cdot \bm{n}] )_{\partial \tau}\\
&= \sum_{\tau\in D_e} \Big((\nabla \cdot \bm{u}_h-f,  \psi_e[\bm{u}_h\cdot \bm{n}])_\tau+(\bm{u}_h-\bm{u}, \nabla  (\psi_e[\bm{u}_h\cdot \bm{n}]))_\tau\Big).
\end{align*}
The following estimate can be achieved by \eqref{extension}, \eqref{eq:tr1}, \eqref{eq:eff-f} and inverse inequality
\begin{equation*}
\begin{split}
\|[\bm{u}_h\cdot \bm{n}]\|_{0,e}^2&\leq C \sum_{\tau\in D_e}  \Big(\|K^{-\frac{1}{2}}(\bm{u}-\bm{u}_h)\|_{0,\tau}
\|K^{\frac{1}{2}} \nabla ( \psi_e[\bm{u}_h\cdot \bm{n}])\|_{0,\tau}+\|f-\nabla \cdot \bm{u}_h\|_{0,\tau}\| \psi_e[\bm{u}_h\cdot \bm{n}]\|_{0,\tau}\Big)\\
&\leq C  \sum_{\tau\in D_e} \Big(\|K^{-\frac{1}{2}}(\bm{u}-\bm{u}_h)\|_{0,\tau}
h_{\tau}^{-1}
\|K^{\frac{1}{2}} \psi_e[\bm{u}_h\cdot \bm{n}]\|_{0,\tau}+\|f-\nabla \cdot \bm{u}_h\|_{0,\tau}\|  \psi_e[\bm{u}_h\cdot \bm{n}]\|_{0,\tau}\Big)\\
&\leq C \Big(\Big(\sum_{\tau\in D_e}\|K^{-\frac{1}{2}}(\bm{u}-\bm{u}_h)\|_{0,\tau}^2
\Big)^{\frac{1}{2}}+\Big(\sum_{\tau\in D_e}h_\tau^2K^{-1}
\|f-f_h\|_{0,\tau}^2\Big)^{\frac{1}{2}}\Big) h_e^{-\frac{1}{2}}\| [\bm{u}_h\cdot \bm{n}]\|_{0,e},
\end{split}
\end{equation*}
which gives \eqref{eq:effujump} by dividing both sides of the above equation by $h_{e}^{-1/2}\|[\bm{u}_h\cdot \bm{n}]\|_{0,e}$.

It remains to estimate \eqref{eq:eff-fracture}.
Let $R_e(f_{\Gamma})=\ell_\Gamma f_\Gamma+\nabla_t \cdot (K_\Gamma\nabla_t p_{\Gamma,h})+[\bm{u}_h\cdot\bm{n}_\Gamma],e\in \mathcal{F}_h^\Gamma$, it then follows from Green's theorem, Cauchy-Schwarz inequality and \eqref{eq:tr1}
\begin{align*}
\langle R_e(f_{\Gamma,h}),\psi_e R_e(f_{\Gamma,h}) \rangle_e&=\langle  -\nabla_t\cdot (K_\Gamma \nabla_t (p_\Gamma-p_{\Gamma,h}),\psi_e R_e(f_{\Gamma,h})\rangle_e-\langle [\bm{u}\cdot\bm{n}_\Gamma-\bm{u}_h\cdot\bm{n}_\Gamma], \psi_e R_e(f_{\Gamma,h}) \rangle_e\\
&\;+\langle \ell_\Gamma (f_{\Gamma,h}-f_\Gamma), \psi_e R_e(f_{\Gamma,h})\rangle_e\\
&=\langle K_\Gamma \nabla_t (p_\Gamma-p_{\Gamma,h}), \nabla_t (\psi_e R_e(f_{\Gamma,h})) \rangle_e-\langle [\bm{u}\cdot\bm{n}_\Gamma-\bm{u}_h\cdot\bm{n}_\Gamma], \psi_e R_e(f_{\Gamma,h}) \rangle_e\\
&\;+\langle \ell_\Gamma (f_{\Gamma,h}-f_\Gamma), \psi_e R_e(f_{\Gamma,h})\rangle_e\\
&\leq \|K_\Gamma \nabla_t (p_\Gamma-p_{\Gamma,h})\|_{0,e}\| \nabla_t (\psi_e R_e(f_{\Gamma,h}))\|_{0,e}\\
&\;+\|[\bm{u}\cdot\bm{n}_\Gamma-\bm{u}_h\cdot\bm{n}_\Gamma]\|_{0,e}\|\psi_e R_e(f_{\Gamma,h})\|_{0,e}+\|\ell_\Gamma (f_{\Gamma,h}-f_\Gamma)\|_{0,e}\|\psi_e R_e(f_{\Gamma,h})\|_{0,e}\\
&\leq C \Big(h_e^{-1}\|K_\Gamma \nabla_t (p_\Gamma-p_{\Gamma,h})\|_{0,e}\|  R_e(f_{\Gamma,h})\|_{0,e}\\
&\;+
\|[\bm{u}\cdot\bm{n}_\Gamma-\bm{u}_h\cdot\bm{n}_\Gamma]\|_{0,e}\|R_e(f_{\Gamma,h})\|_{0,e}
+\|\ell_\Gamma (f_{\Gamma,h}-f_\Gamma)\|_{0,e}\|R_e(f_{\Gamma,h})\|_{0,e}\Big).
\end{align*}
Therefore
\begin{align*}
\|R_e(f_{\Gamma,h})\|_{0,e}^2&\leq C \Big(h_e^{-1}\|K_\Gamma \nabla_t (p_\Gamma-p_{\Gamma,h})\|_{0,e}\|  R_e(f_{\Gamma,h})\|_{0,e}
+\|[\bm{u}\cdot\bm{n}_\Gamma-\bm{u}_h\cdot\bm{n}_\Gamma]\|_{0,e}\|R_e(f_{\Gamma,h})\|_{0,e}\\
&\;+\|\ell_\Gamma (f_{\Gamma,h}-f_\Gamma)\|_{0,e}\|R_e(f_{\Gamma,h})\|_{0,e}\Big),
\end{align*}
which leads to \eqref{eq:eff-fracture}.

\end{proof}

\begin{lemma}\label{lemma:interfacej}

Let $(\bm{u}_h,p_h,p_{\Gamma,h})$ be the discrete solution of \eqref{eq:discrete}, then we have for $e\in \mathcal{F}_h^\Gamma$
\begin{align*}
h_e^{1/2}\|(\{p_h\}-p_{\Gamma,h})-\alpha_\Gamma[\bm{u}_h]\|_{0,e}&\leq C\Big(h_e^{1/2}\|\alpha_\Gamma^{-1/2}(\{p-p_h\}-(p_\Gamma-p_{\Gamma,h}))\|_{0,e}
+h_e^{1/2}\|[(\bm{u}-\bm{u}_h)\cdot\bm{n}_\Gamma]\|_{0,e}\Big),\\
h_e^{1/2}\|\eta_\Gamma\{\bm{u}_h\cdot\bm{n}_\Gamma\}-[p_h]\|_{0,e}&\leq C\Big(h_e^{1/2}\|\eta_\Gamma^{-1/2}[p-p_h]\|_{0,e}
+h_e^{1/2}\|\{(\bm{u}-\bm{u}_h)\cdot\bm{n}_\Gamma\}\|_{0,e}\Big).
\end{align*}
\end{lemma}

\begin{proof}
The desired estimates follows directly from triangle inequality and the interface conditions \eqref{eq:interface}.
\end{proof}

\begin{lemma}\label{lemma:eff2}
Let $p_{\Gamma,h}$ be the discrete solution of \eqref{eq:discrete} and $f_{\Gamma,h}$ be the piecewise polynomial approximation of $f_\Gamma$, then we have for any $z\in \mathcal{N}_h^0(\Gamma)$
\begin{align*}
h_z^{1/2}|[K_\Gamma^{1/2}\nabla_tp_{\Gamma,h}]|_z&\leq C\Big(\|K_\Gamma^{1/2} \nabla_t (p_\Gamma-p_{\Gamma,h})\|_{0,e_z}+h_z\|[\bm{u}\cdot\bm{n}_\Gamma-\bm{u}_h\cdot\bm{n}_\Gamma]\|_{0,e_z}+h_z\|\ell_\Gamma (f_{\Gamma,h}-f_\Gamma)\|_{0,e_z}\Big),
\end{align*}
where $e_z$ represents the two edges belonging to $\mathcal{F}_h^\Gamma$ sharing the common vertex $z$.

\end{lemma}

\begin{proof}

We use $\psi_z$ to stand for the vertex bubble function associated to the vertex $z$. Then, by equivalence of norms on finite-dimensional spaces, there holds
\begin{align}
[\nabla_tp_{\Gamma,h}]_z  [\nabla_tp_{\Gamma,h}]_z\leq C [\nabla_tp_{\Gamma,h}]_z\psi_z[\nabla_tp_{\Gamma,h}]_z.\label{eq:gradzf}
\end{align}
We keep the same notation for the constant extension of the function $[K_\Gamma\nabla_tp_{\Gamma,h}]$, originally only defined on the vertex $z$, to a function defined on the edge $e_z$.
Therefore, it holds
\begin{align}
\|[\nabla_tp_{\Gamma,h}]\|_{e_z}\leq C h_z^{1/2}|[\nabla_tp_{\Gamma,h}]_z|.\label{eq:fez}
\end{align}
In addition, we also have
\begin{align}
\|\psi_z [\nabla_t p_{\Gamma,h}]\|_{0,e_z}\leq C \|[\nabla_t p_{\Gamma,h}]\|_{0,e_z}.\label{eq:f1}
\end{align}
It follows from Green's theorem, inverse inequality and \eqref{eq:gradzf}-\eqref{eq:f1}
\begin{align*}
C|[K_\Gamma^{1/2}\nabla_tp_{\Gamma,h}]_z|^2&\leq [K_\Gamma\nabla_tp_{\Gamma,h}]_z\psi_z ([\nabla_tp_{\Gamma,h}]_z)\\
&=\langle \ell_{\Gamma}f_\Gamma+\nabla_t \cdot (K_\Gamma\nabla_t p_{\Gamma,h})+[\bm{u}\cdot \bm{n}_\Gamma], \psi_z ([\nabla_tp_{\Gamma,h}])\rangle_{e_z}\\
&\;+\langle \nabla_t K_\Gamma(p_{\Gamma,h}-p_\Gamma),\nabla_t (\psi_z ([\nabla_tp_{\Gamma,h}])) \rangle_{e_z}\\
&\leq C \Big(\|\ell_{\Gamma}f_\Gamma+\nabla_t \cdot (K_\Gamma \nabla_t p_{\Gamma,h})+[\bm{u}_h\cdot \bm{n}_\Gamma]\|_{0,e_z}\|\psi_z ([\nabla_tp_{\Gamma,h}])\|_{0,e_z}\\
&\;+\|[\bm{u}\cdot\bm{n}_\Gamma-\bm{u}_h\cdot\bm{n}_\Gamma]\|_{0,e_z}\|\psi_z ([\nabla_tp_{\Gamma,h}])\|_{0,e_z}\\
&\;+\|\nabla_t K_\Gamma(p_{\Gamma,h}-p_\Gamma)\|_{0,e_z}\|\nabla_t (\psi_z ([\nabla_tp_{\Gamma,h}]))\|_{0,e_z}\Big)\\
&\leq C \Big(\|\ell_{\Gamma}f_\Gamma+\nabla_t \cdot (K_\Gamma \nabla_t p_{\Gamma,h})+[\bm{u}_h\cdot \bm{n}_\Gamma]\|_{0,e_z}\| [\nabla_tp_{\Gamma,h}]\|_{0,e_z}\\
&\;+\|[\bm{u}\cdot\bm{n}_\Gamma-\bm{u}_h\cdot\bm{n}_\Gamma]\|_{0,e_z}\| [\nabla_tp_{\Gamma,h}]\|_{0,e_z}
+h_z^{-1}\|\nabla_t K_\Gamma(p_{\Gamma,h}-p_\Gamma)\|_{0,e_z}\| [\nabla_tp_{\Gamma,h}]\|_{0,e_z}\Big)\\
&\leq C\Big(h_z^{1/2}\|\ell_{\Gamma}f_\Gamma+\nabla_t \cdot (K_\Gamma\nabla_t p_{\Gamma,h})+[\bm{u}_h\cdot \bm{n}_\Gamma]\|_{0,e_z} |[\nabla_tp_{\Gamma,h}]_z|\\
&\;+h_z^{1/2}\|[\bm{u}\cdot\bm{n}_\Gamma-\bm{u}_h\cdot\bm{n}_\Gamma]\|_{0,e_z}|[\nabla_tp_{\Gamma,h}]_z|
+h_z^{-1/2}\|\nabla_t K_\Gamma(p_{\Gamma,h}-p_\Gamma)\|_{0,e_z}| [\nabla_tp_{\Gamma,h}]_z|\Big).
\end{align*}
The desired estimate holds by applying \eqref{eq:eff-fracture}.

\end{proof}

\begin{theorem}

Combining Lemma~\ref{lemma:eff1}, Lemma~\ref{lemma:interfacej} and Lemma~\ref{lemma:eff2}, we can obtain
\begin{align*}
\eta\leq C \|(\bm{u}-\bm{u}_h,p-p_h,p_\Gamma-p_{\Gamma,h})\|_{sdg}+\mbox{osc}(f,f_\Gamma),
\end{align*}
where $\mbox{osc}(f,f_\Gamma)$ is the data oscillation and is defined by $\mbox{osc}(f,f_\Gamma)^2=\sum_{e\in \mathcal{F}_h^\Gamma}h_e^2\|\ell_\Gamma (f_{\Gamma,h}-f_\Gamma)\|_{0,e}^2+\sum_{\tau\in \mathcal{T}_h}h_\tau^2
\|f-f_h\|_{0,\tau}^2$.

\end{theorem}



\section{Numerical experiments}\label{sec:numerical}

In this section we present several numerical experiments to verify the accuracy and efficiency of the proposed error estimators. We first consider fractured porous media domain with one single fracture, then to further indicate that our method can be applied to multiple fractures, we also carry out numerical experiments for fractured porous media with multiple non-intersecting fractures. The adaptive mesh refinement algorithm can be referred to \cite{LinaPark}. Notice that our method can handle fairly general meshes and hanging nodes can be simply incorporated into the construction of the method, which greatly simplifies the refinement procedure.

In the following examples, we set $K=Id$ and the thickness of the fracture is defined by $\ell_\Gamma=0.01$, where $Id$ is the two dimensional identity matrix. In addition we set $\xi=3/4$.

\subsection{Single fracture on rectangular domain}\label{example1}

We consider $\Omega_B = (0,2)\times (0,1)$ with only one fracture $\Gamma=\{1\}\times (0,1)$. Let $\Omega_1=(0,1)^2$ and $\Omega_2=(1,2)\times (0,1)$. In the first case, we consider the model problem with exact pressure solution $p\mid_{\Omega_1}=y+1/2\tanh(\frac{x-1}{\alpha})+1/2$, $p\mid_{\Omega_2}=y+1/2\tanh(\frac{x-1}{2\alpha})+1/2+\frac{3\eta_\Gamma}{8\alpha}$, $p_\Gamma=y+1/2+\frac{3\eta_\Gamma}{16\alpha}+\frac{\alpha_\Gamma}{4\alpha}$, where $\alpha$ is an optional parameter.
In addition, we define $\kappa_\Gamma^n=\kappa_\Gamma^*=100$. We enforce Dirichlet boundary conditions for both the surrounding porous media boundary and the fracture boundaries.

The convergence history for both the error estimator and the total error against the number of degrees of freedom ($N$) for the polynomial order $k=1,2$ is reported in Figure~\ref{ex1-case1-con}. We can observe that expected convergence rates $\mathcal{O}(N^{-k/2})$ can be achieved for $k=1,2$ with $\alpha=0.1$ and $\alpha=0.01$. Moreover, we also display the effectivity index $\mbox{EI}:=\frac{\eta}{\|(\bm{u}-\bm{u}_h,p-p_h, p_\Gamma-p_{\Gamma,h})\|_{sdg}}$ in Figure~\ref{ex1-case1-con}, which shows that the effectivity index lay in the range of $1.4-1.6$ for different values of $\alpha$. The adaptive mesh pattern and the corresponding numerical approximation for pressure for $k=2$ and $\alpha=0.01$ are shown in Figure~\ref{ex1-case1-mesh}. We can see that the mesh is locally refined near the fracture $\Gamma$ due to the fact that a transition layer is introduced for the pressure $p$ near $\Gamma$ as one can see from the numerical approximation for pressure (cf. Figure~\ref{ex1-case1-mesh}).

In the second case, we consider the model problem without exact solutions. The normal permeability in the fracture is define by: $\kappa_\Gamma^n=\kappa_\Gamma^*=200$ on $\Gamma_1=\{1\}\times ([0,1/4]\cup[3/4,1])$ and $\kappa_\Gamma^n=\kappa_\Gamma^*=0.002$ on $\Gamma_2=\{1\}\times [1/4,3/4]$. We impose homogeneous Nerumann boundary conditions for the fracture boundaries. For the surrounding porous media, the Dirichlet boundary condition is given by $p=1$ on $\{2\}\times [0,1]$ and $p=0$ on $\{0\}\times [0,1]$, and the remaining part of $\partial \Omega$ is homogeneous Neumann boundary condition.

The adaptive mesh pattern is reported in Figure~\ref{ex1-case2}, and we can see that the mesh is locally refined near the ends of $\Gamma_2$. This is because of the fact that the permeability for $\Gamma_1$ and $\Gamma_2$ is different and $\Gamma_2$ represents a barrier due to the small permeability. This is consistent with the numerical approximation for pressure (cf. Figure~\ref{ex1-case2}), where pressure is discontinuous across $\Gamma_2$. Again, we show the convergence history for $\eta$ against the number of degrees of freedom for $k=1$ and $k=2$ under uniform refinement and adaptive refinement in Figure~\ref{ex1-case2}. We can observe that optimal convergence rates can be achieved under adaptive refinement, thus we can conclude that adaptive mesh refinement outperforms uniform mesh refinement.

\begin{figure}[t]
    \centering
    \includegraphics[width=0.32\textwidth]{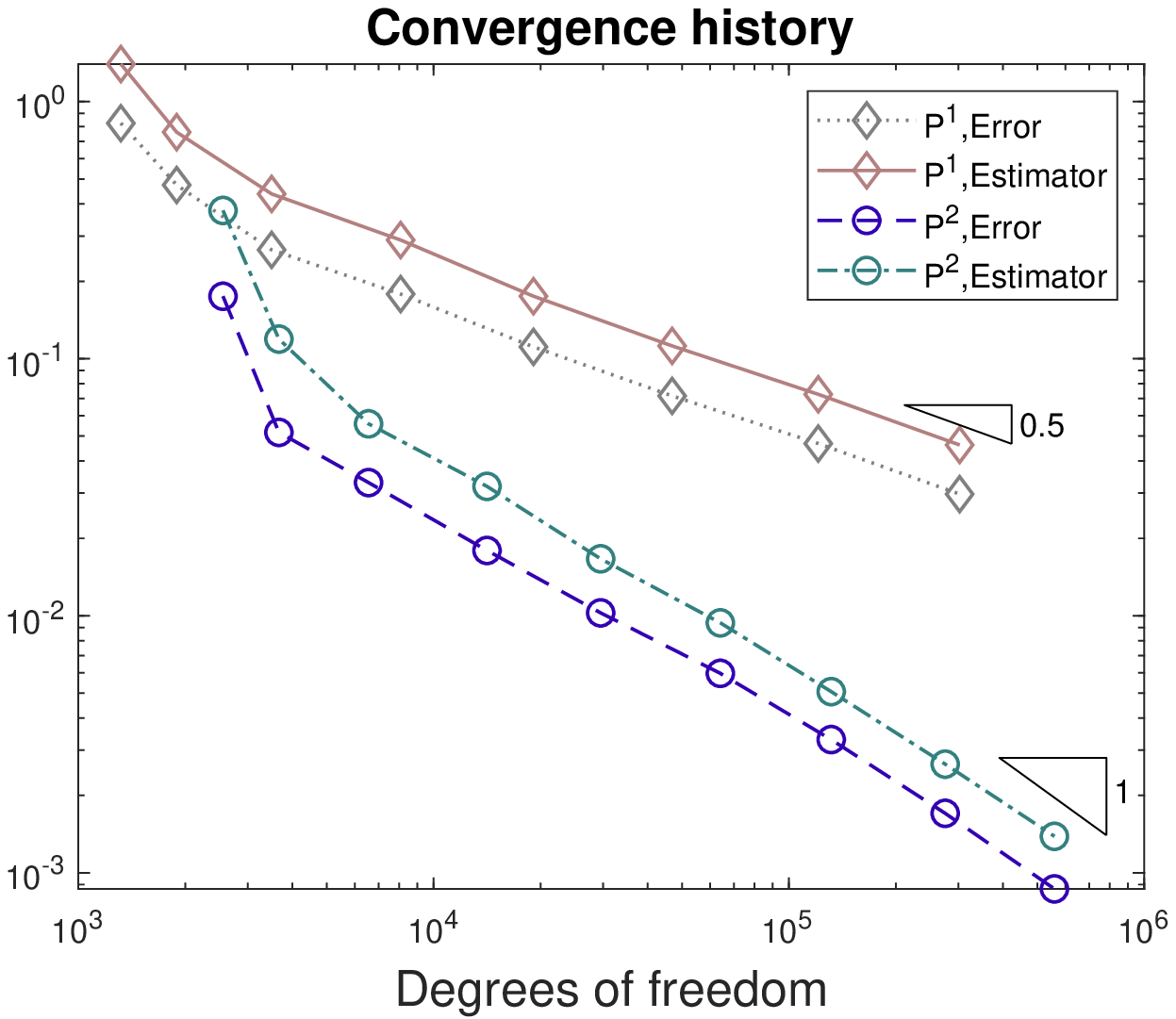}
    \includegraphics[width=0.32\textwidth]{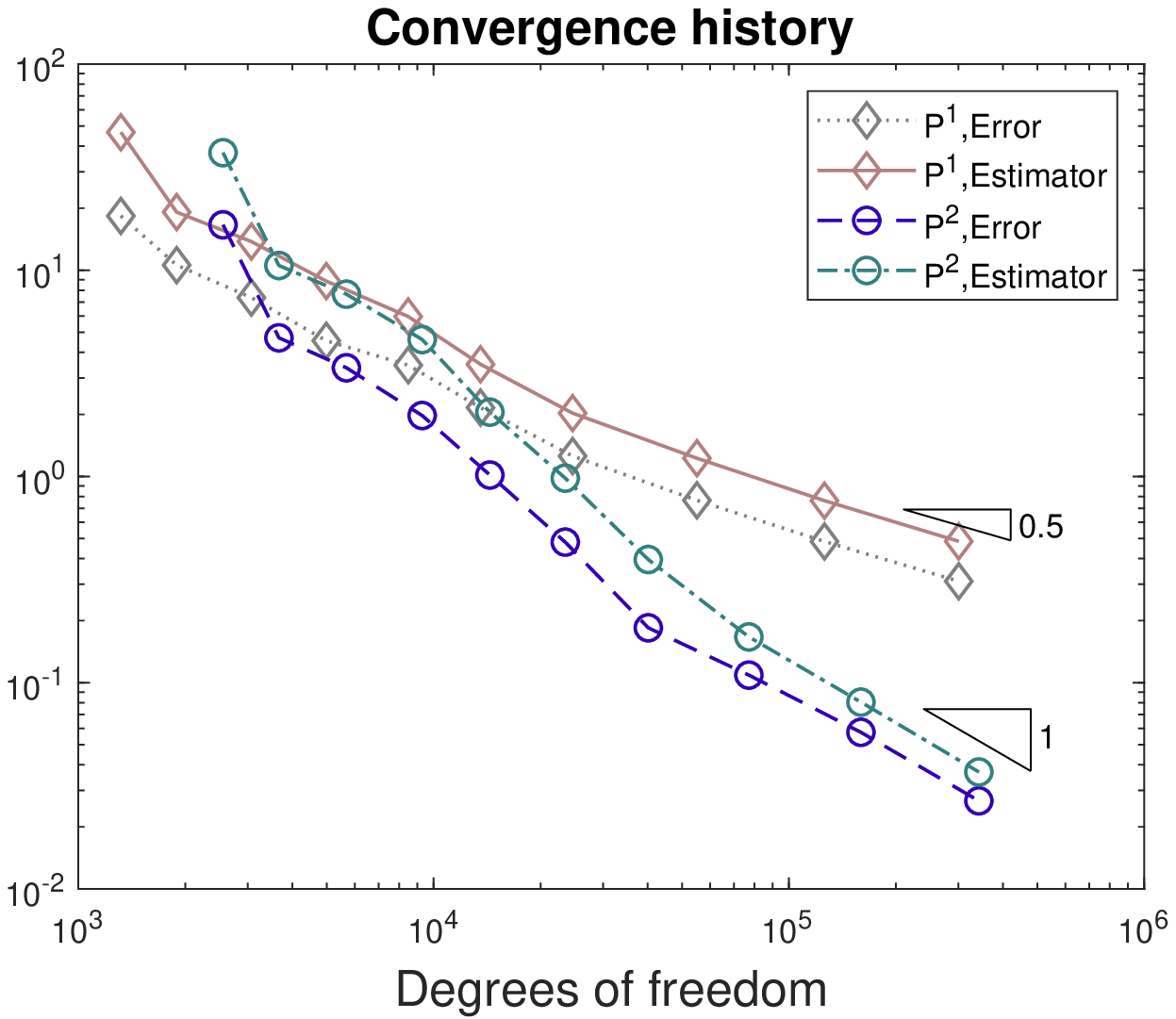}
     \includegraphics[width=0.32\textwidth]{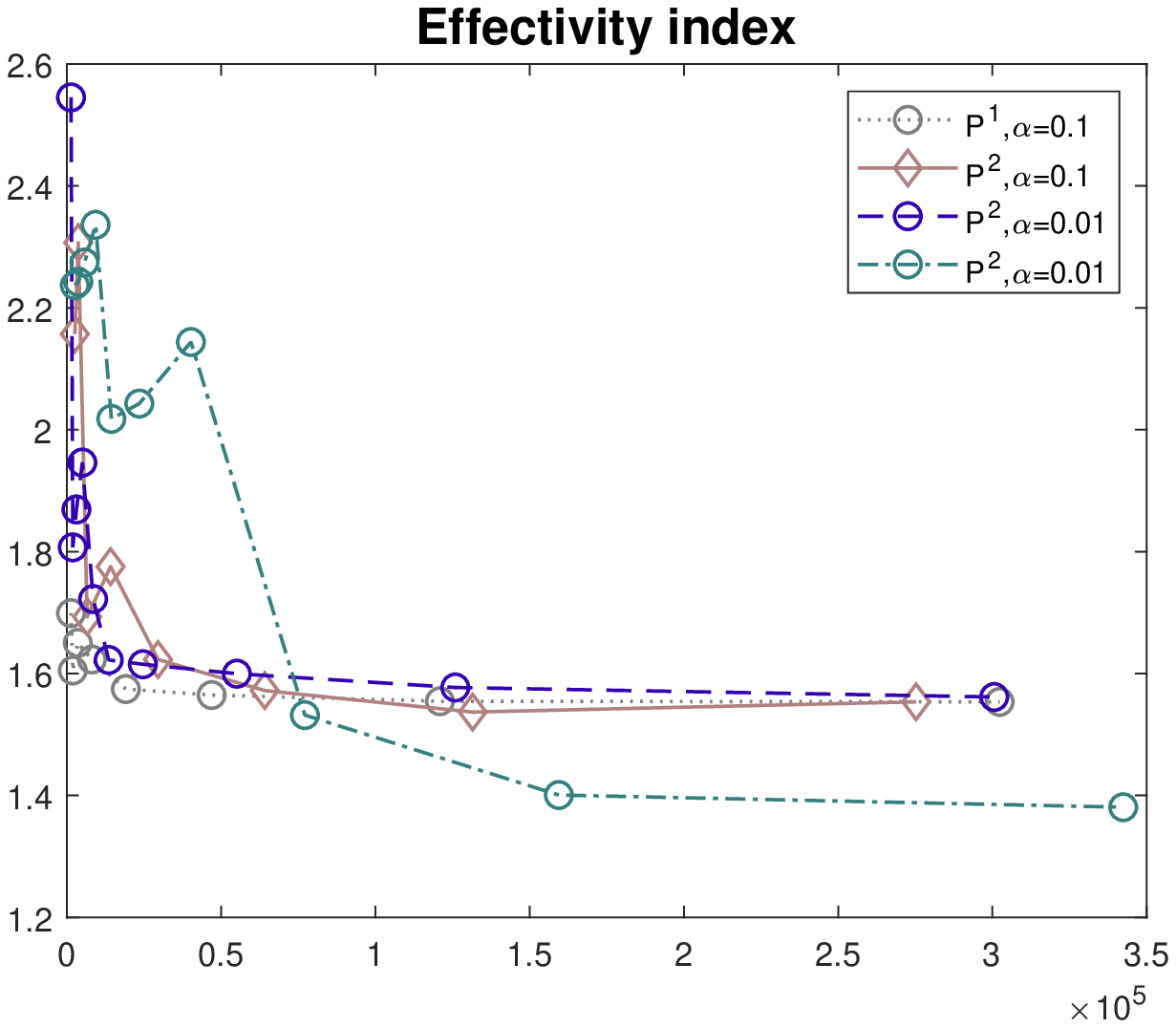}
    \caption{Convergence history of the a posteriori error estimator $\eta$ and the error $\|(\bm{u}-\bm{u}_h,p-p_h, p_\Gamma-p_{\Gamma,h})\|_{sdg}$. Left, $\alpha=0.1$. Middle, $\alpha=0.01$. Right, effectivity index for both cases for Example~\ref{example1}.}
    \label{ex1-case1-con}
\end{figure}

\begin{figure}[t]
    \centering
    \includegraphics[width=0.35\textwidth]{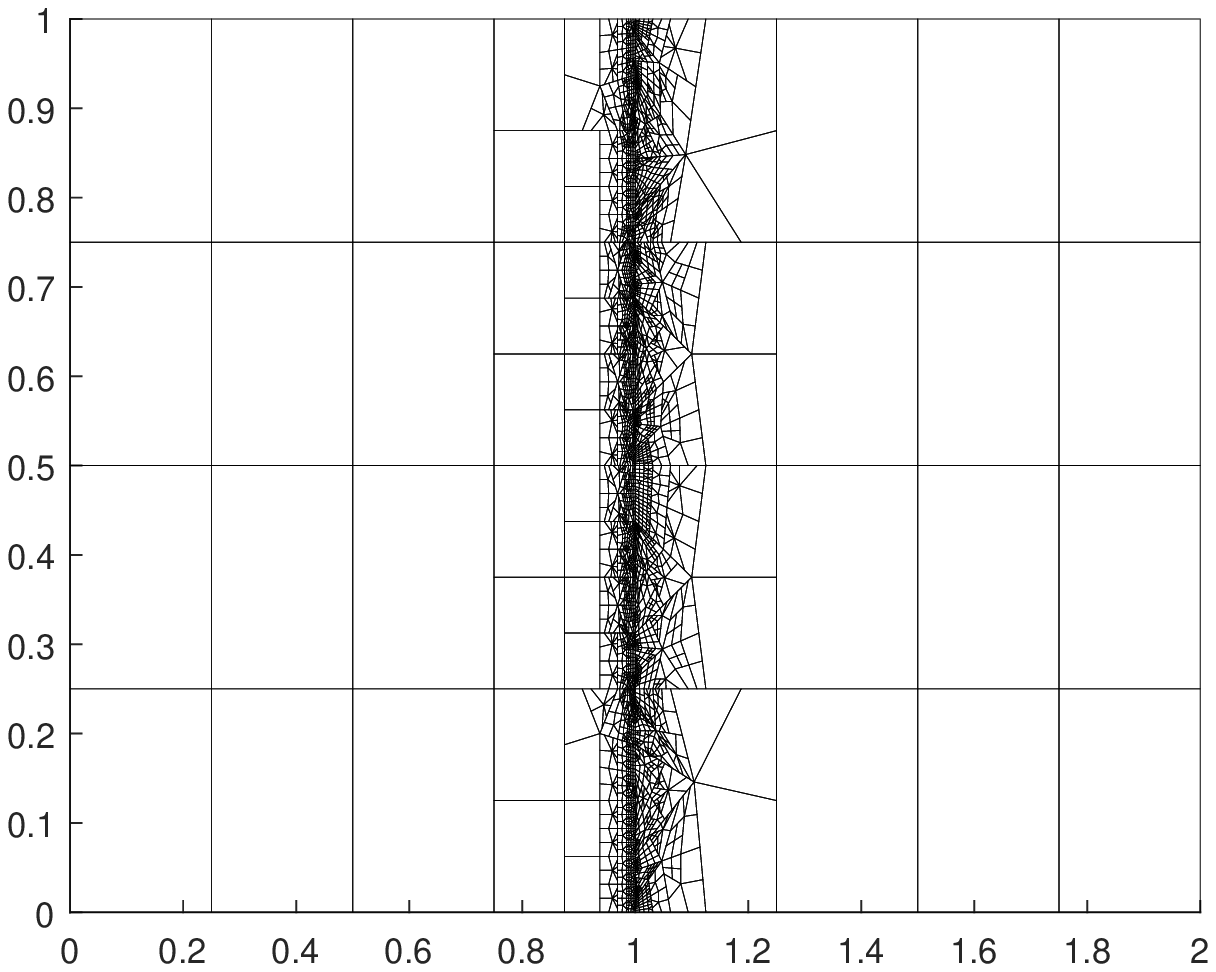}
    \includegraphics[width=0.35\textwidth]{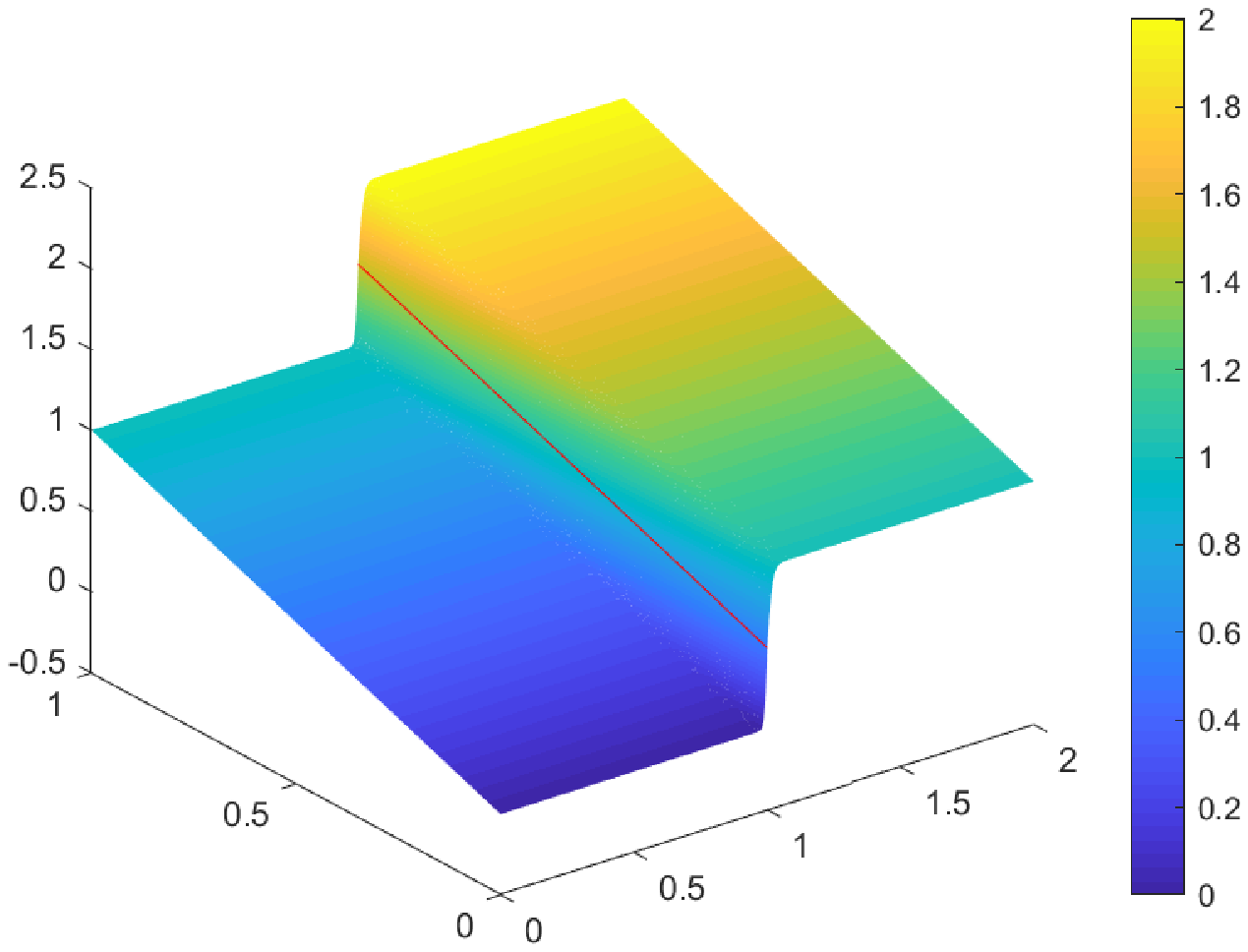}
    \caption{Adaptive mesh patter for $k=2, \alpha=0.01$ (left) and the corresponding numerical approximation for pressure (right) for the first case of Example~\ref{example1}.}
    \label{ex1-case1-mesh}
\end{figure}

\begin{figure}[t]
    \centering
    \includegraphics[width=0.32\textwidth]{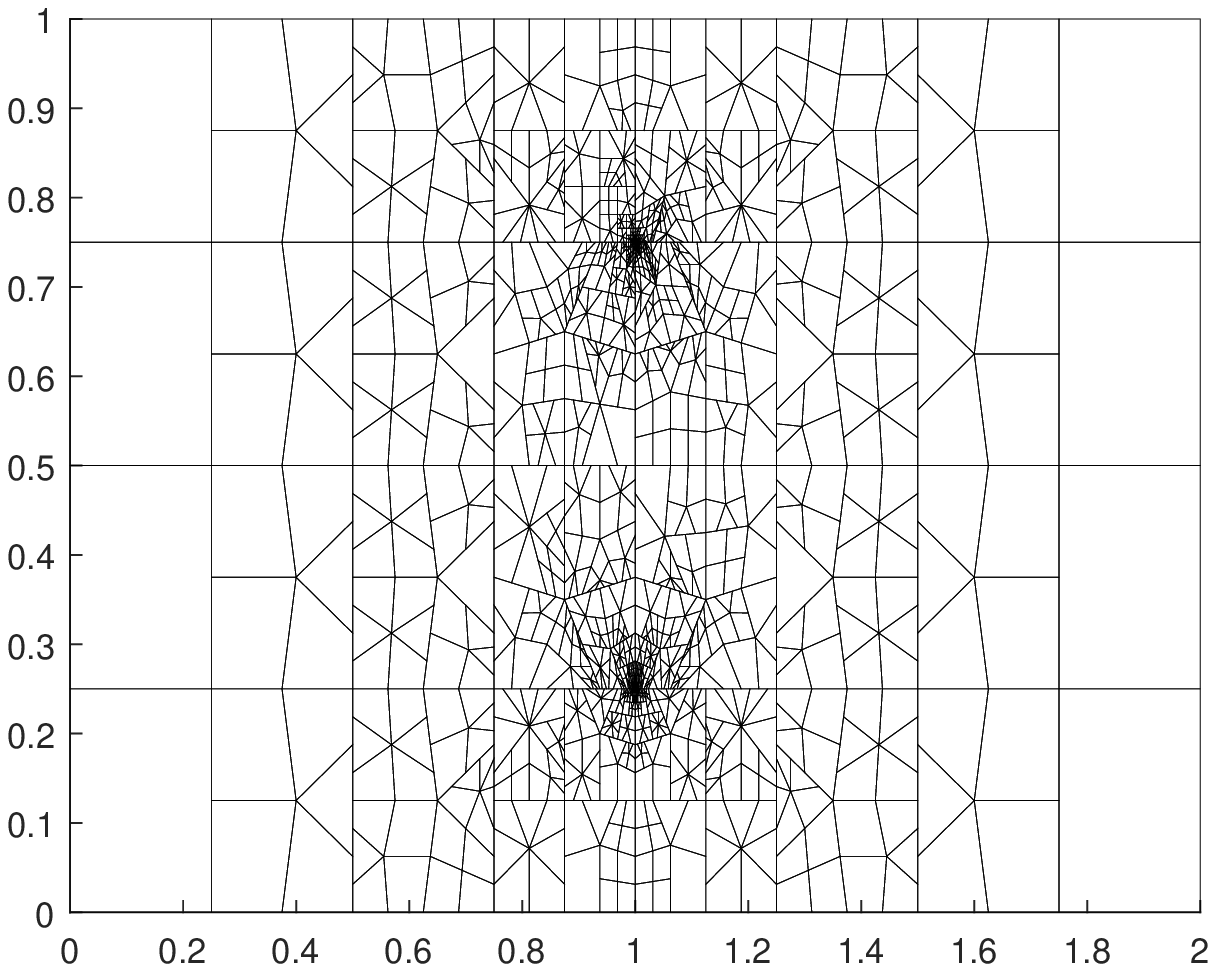}
    \includegraphics[width=0.32\textwidth]{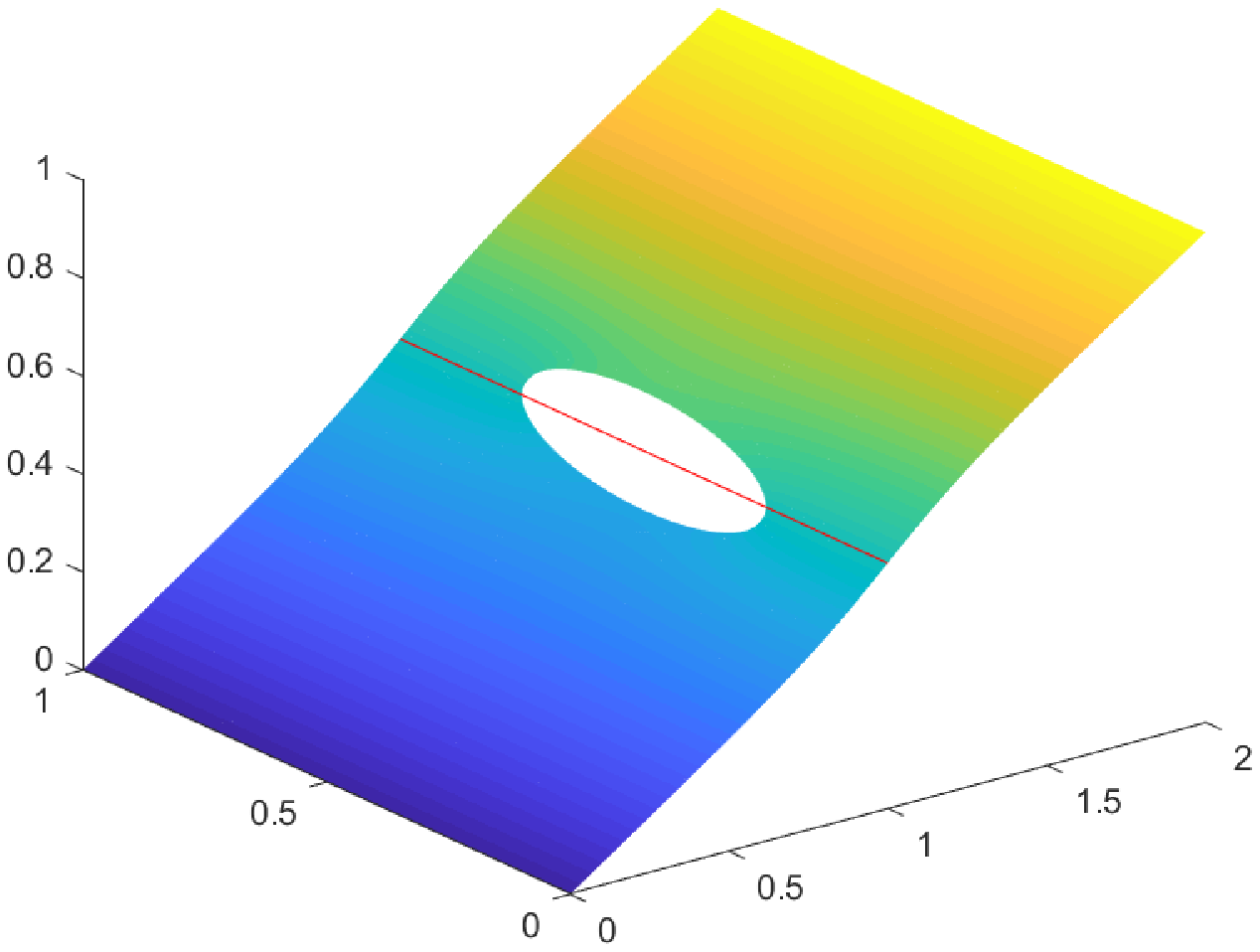}
    \includegraphics[width=0.32\textwidth]{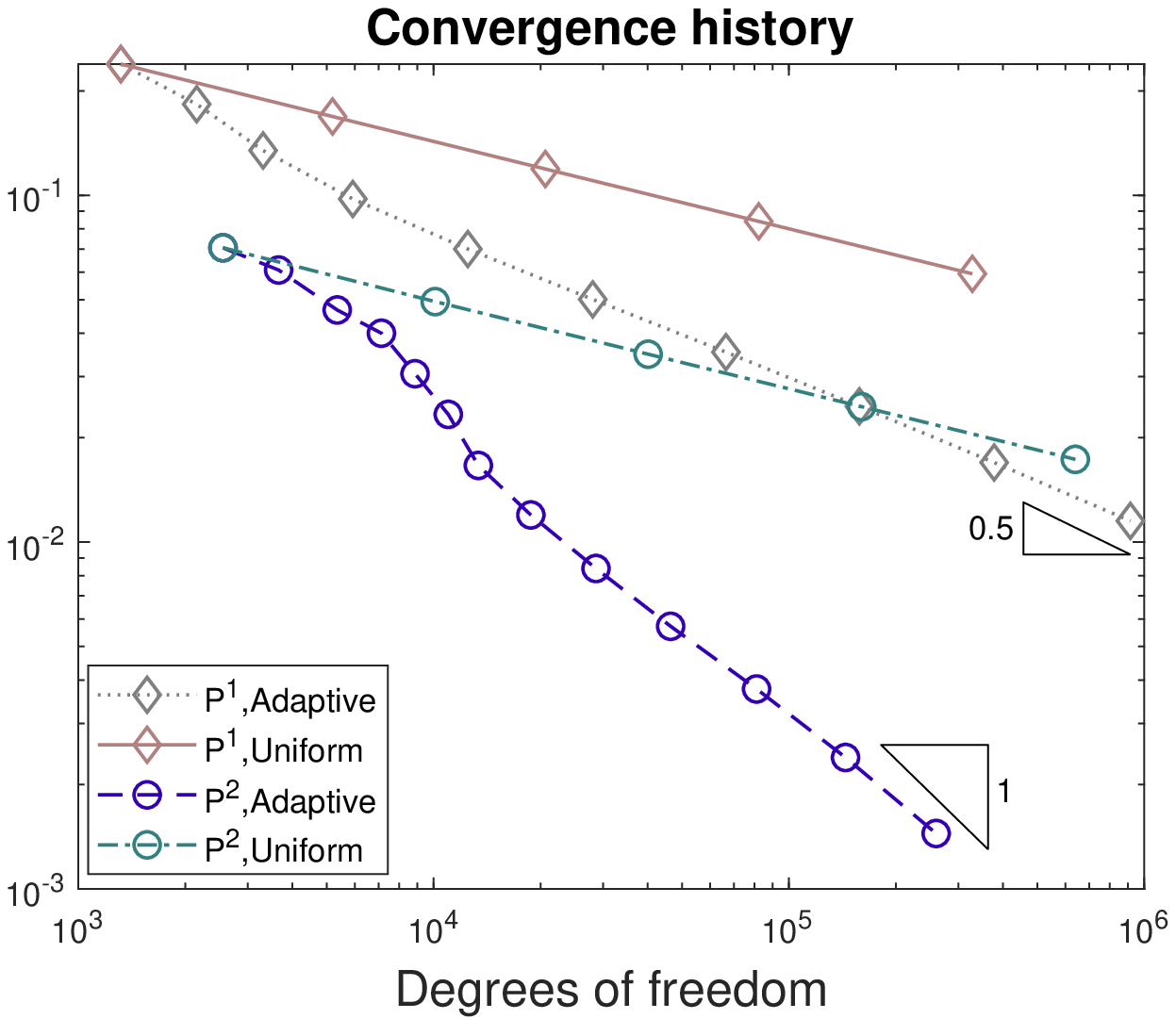}
    \caption{Adaptive mesh patter for $k=2$ (left), the corresponding numerical approximation for pressure (middle) and the convergence history for $\eta$ under unform mesh refinement and adaptive mesh refinement for the second case of Example~\ref{example1}.}
    \label{ex1-case2}
\end{figure}


\subsection{Single fracture on $L$-shaped domain}\label{example2}

We consider the $L$-shaped domain $\Omega_B = ([0,2]\times[-1,1])\backslash ([0,1]\times[-1,0])$. The fractures lie on the polygonal line $\Gamma_1\cup\Gamma_2\cup\Gamma_3$, where $\Gamma_1=(\{1/2\}\times[1/2,1])\cup([1/2,1]\times\{1/2\})$, $\Gamma_2=([1,3/2]\times \{1/2\})\cup(\{3/2\}\times[0,1/2])$ and $\Gamma_3=\{3/2\}\times [-1,0]$. The normal permeability in the fracture is given by $\kappa_\Gamma^n=\kappa_\Gamma^*=100$ on $\Gamma_1\cup\Gamma_3$ and $\kappa_\Gamma^n=\kappa_\Gamma^*=0.001$ on $\Gamma_2$. For the surrounding porous media, the Dirichlet boundary condition is given by $p=1$ on $[0,2]\times \{1\}$ and $p=0$ on $[1,2]\times \{-1\}$, and homogeneous Neumann boundary condition is imposed on the remaining part of $\partial \Omega_B$. In addition, we impose Dirichlet boundary condition for the fracture boundaries, where $p_\Gamma=1$ at $(1/2,1)$ and $p_\Gamma=0$ at $(3/2,-1)$.

The numerical approximation for pressure (cf. Figure~\ref{ex2-result}) experiences jump across $\Gamma_2$ due to the low permeability. The mesh is locally refined near $\Gamma_2$, the corner point $(1,0)$ and the two end points of $\Gamma_2$, see Figure~\ref{ex2-result}. Moreover, we display the convergence history against the number of degrees of freedom for uniform refinement and adaptive refinement, as expected, optimal convergence rates can be recovered by the adaptive mesh refinement. Here we only show the numerical results for $k=1$ for the sake of simplicity. The superiority of adaptive mesh refinement can be easily observed from the convergence history.

\begin{figure}[t]
    \centering
    \includegraphics[width=0.32\textwidth]{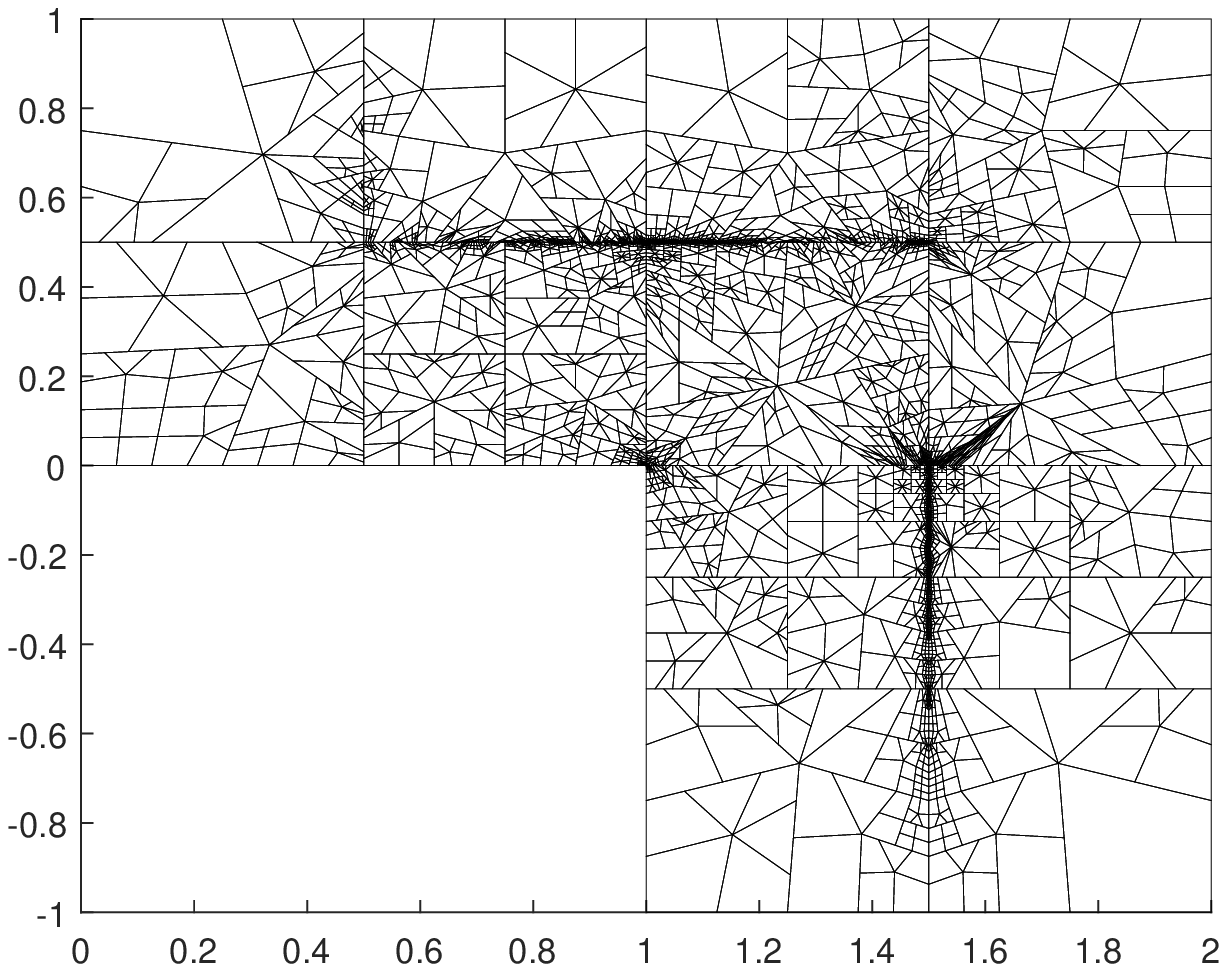}
    \includegraphics[width=0.32\textwidth]{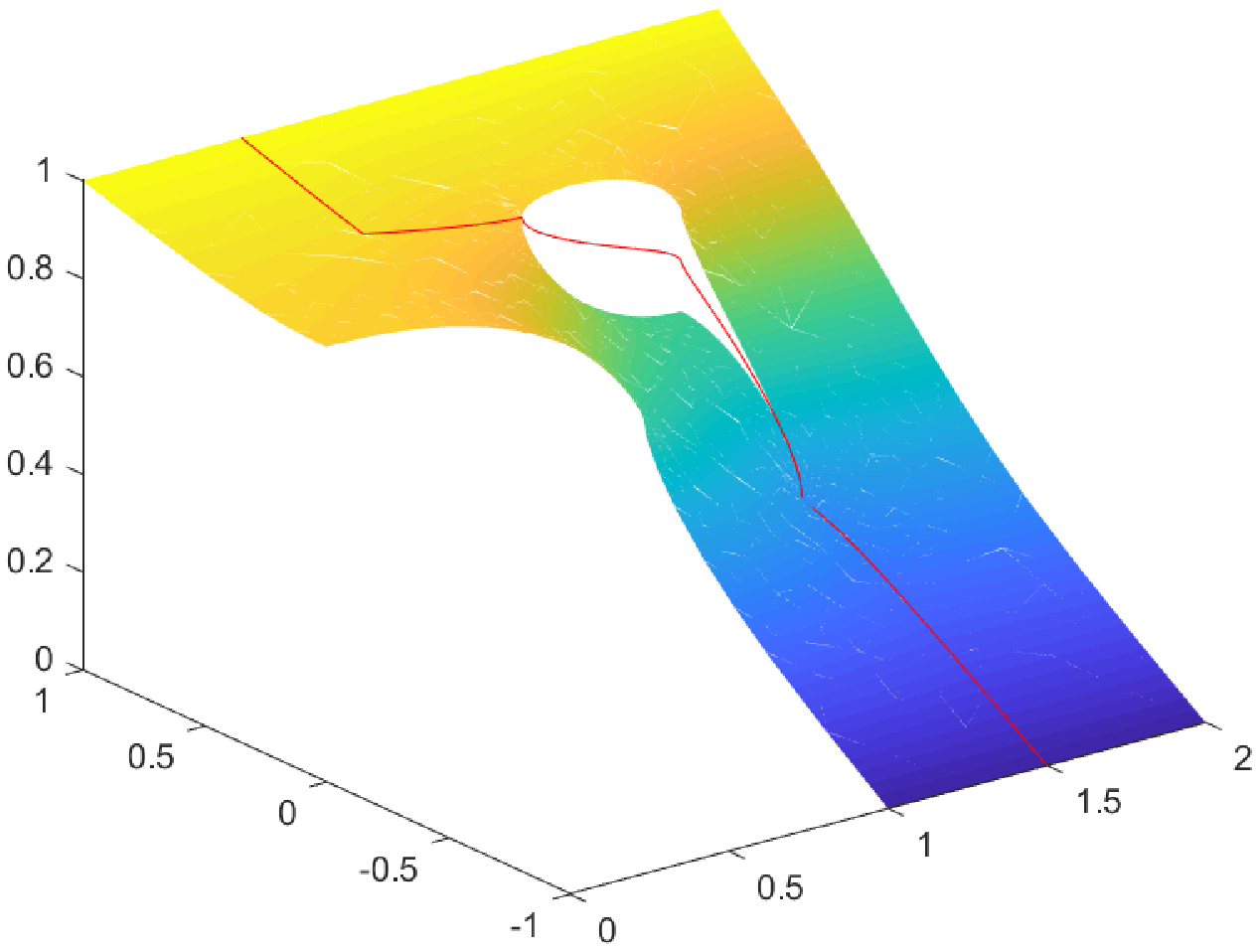}
     \includegraphics[width=0.32\textwidth]{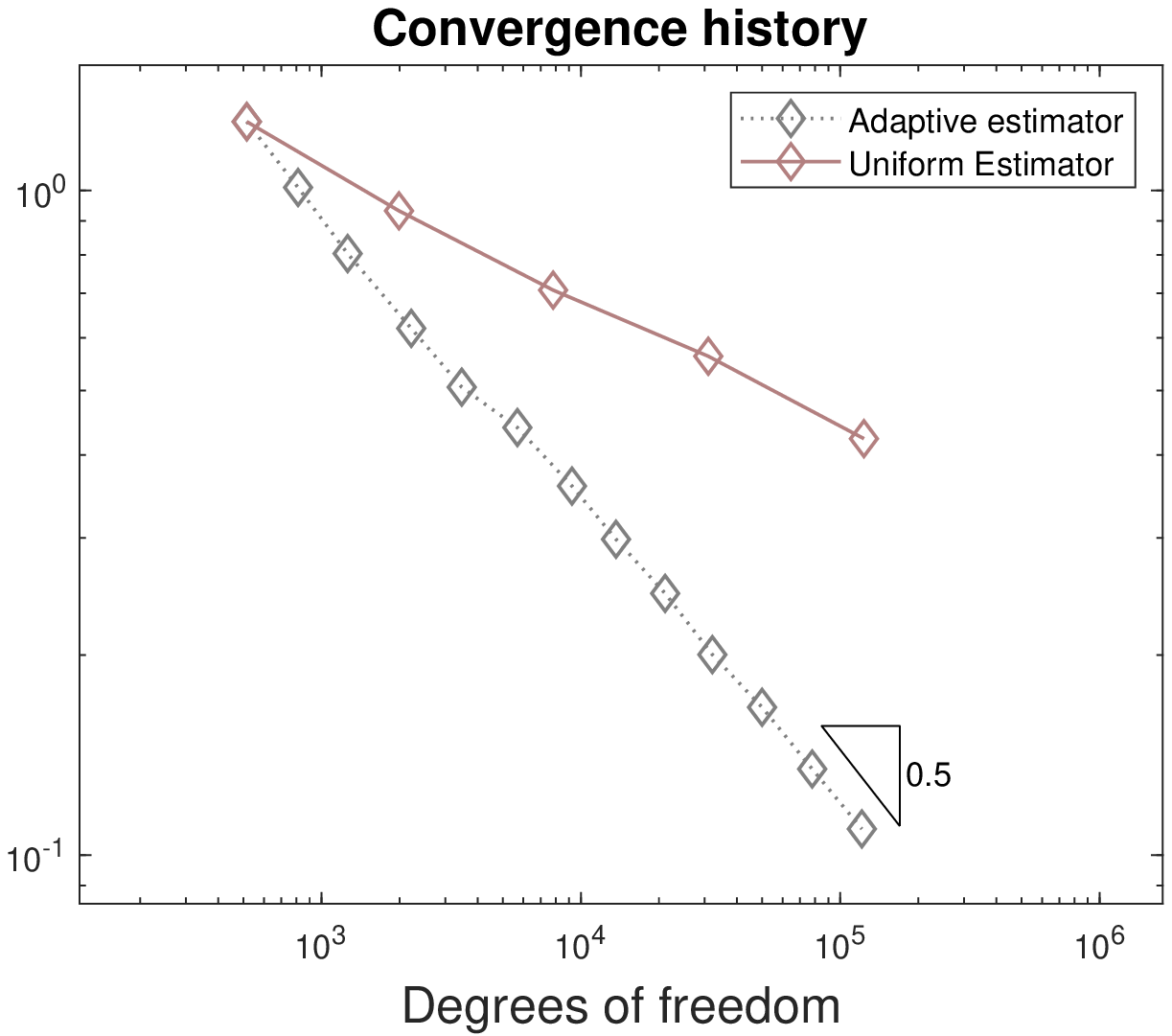}
    \caption{Adaptive mesh patter (left), the corresponding numerical approximation for pressure (middle) and the convergence history for Example~\ref{example2}.}
    \label{ex2-result}
\end{figure}

\subsection{Multiple non-intersecting fractures on $L$-shaped domain}\label{example3}

In this example, we consider more fractures totally or partially immersed in the fractured porous media. We again use the $L$-shaped domain defined in Example~\ref{example2} and the fractures lie on the lines $\Gamma_1,\Gamma_2,\Gamma_3$ and $\Gamma_4$, where $\Gamma_1=[1/2,1]\times \{1/2\}$, $\Gamma_2=\{3/2\}\times [1/2,1]$, $\Gamma_3=[3/2,2]\times \{0\}$ and $\Gamma_4=\{3/2\}\times [-1,-1/2]$. The permeability in the fracture is given by $\kappa_\Gamma^n=\kappa_\Gamma^*=100$ on $\Gamma_1$ and $\Gamma_4$, $\kappa_\Gamma^n=\kappa_\Gamma^*=0.001$ on $\Gamma_2$ and $\kappa_\Gamma^n=\kappa_\Gamma^*=0.01$ on $\Gamma_3$. For the surrounding porous media domain, we enforce Dirichlet boundary condition by $p=1$ on $\{0\}\times[0,1]$ and $p=0$ on $[1,2]\times \{-1\}$. In addition, homogeneous Dirichlet boundary condition is imposed at the point $(3/2,-1)$ for the fracture pressure and the remaining fracture boundaries are equipped by homogeneous Neumann boundary condition.

The numerical approximation for pressure is reported in Figure~\ref{ex3}, where we can see that the pressure is discontinuous across $\Gamma_2$ and $\Gamma_3$ due to the low permeability. The mesh is locally refined near the cornet point $(1,0)$, the interior end points of the fractures, in addition, locally refined mesh can also be observed across $\Gamma_2$ and $\Gamma_3$, which is caused by the discontinuity of the pressure. The convergence history against the number of degrees of freedom for $k=1$ under uniform refinement and adaptive refinement is displayed in Figure~\ref{ex3}, and optimal convergence rates can be recovered by adaptive mesh refinement. This example once again highlights that adaptive mesh refinement outperforms uniform mesh refinement.

\begin{figure}[t]
    \centering
    \includegraphics[width=0.32\textwidth]{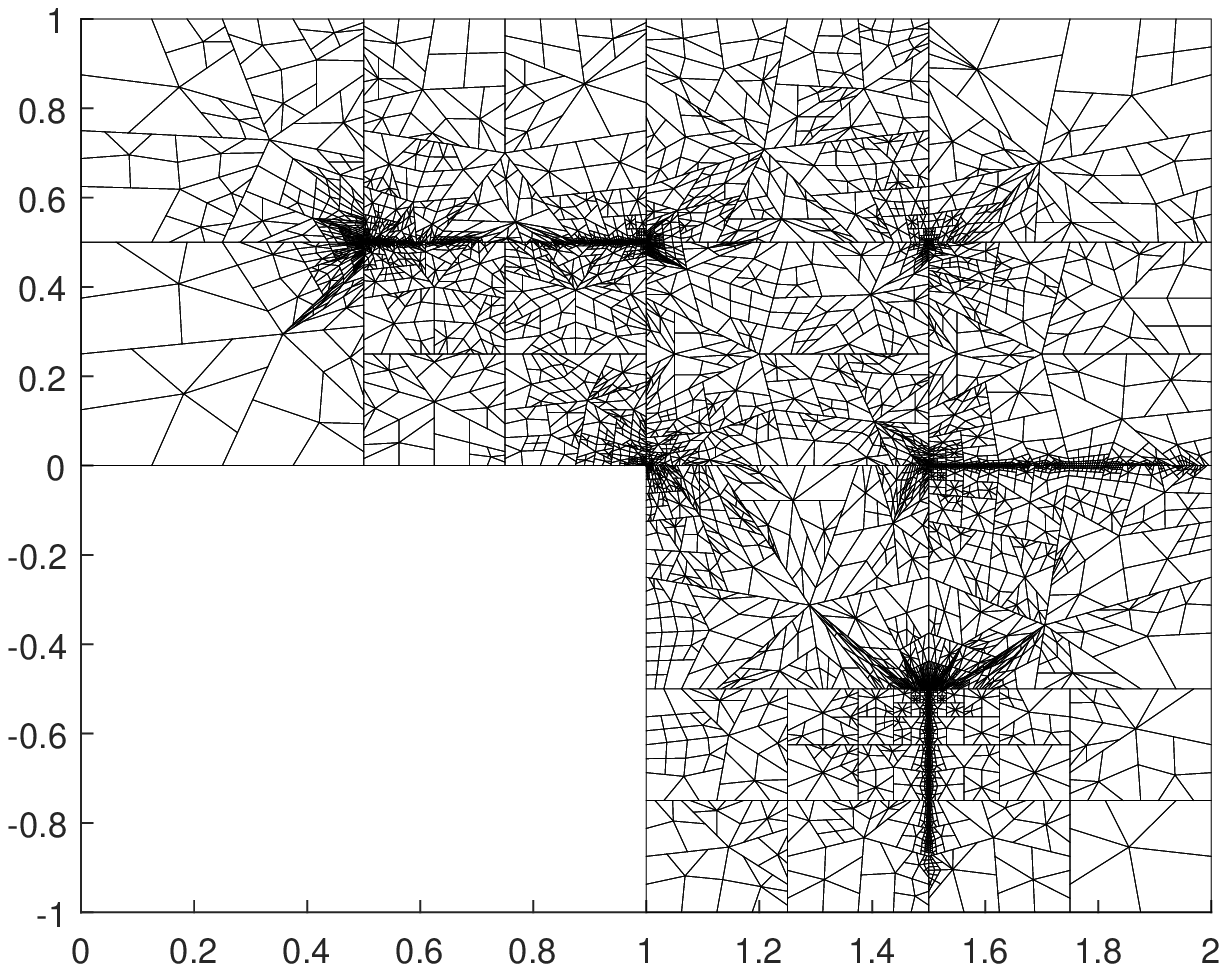}
    \includegraphics[width=0.32\textwidth]{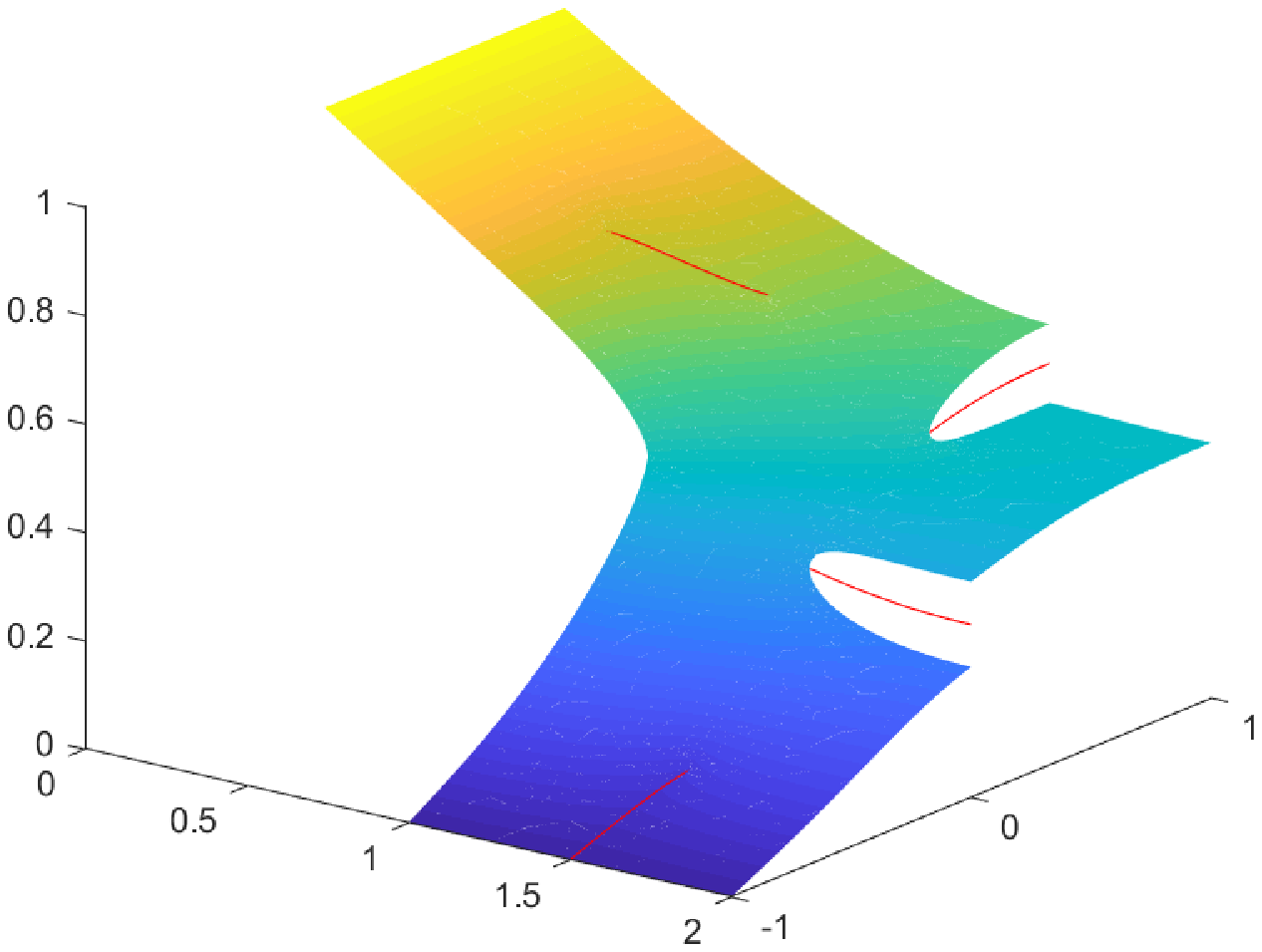}
    \includegraphics[width=0.32\textwidth]{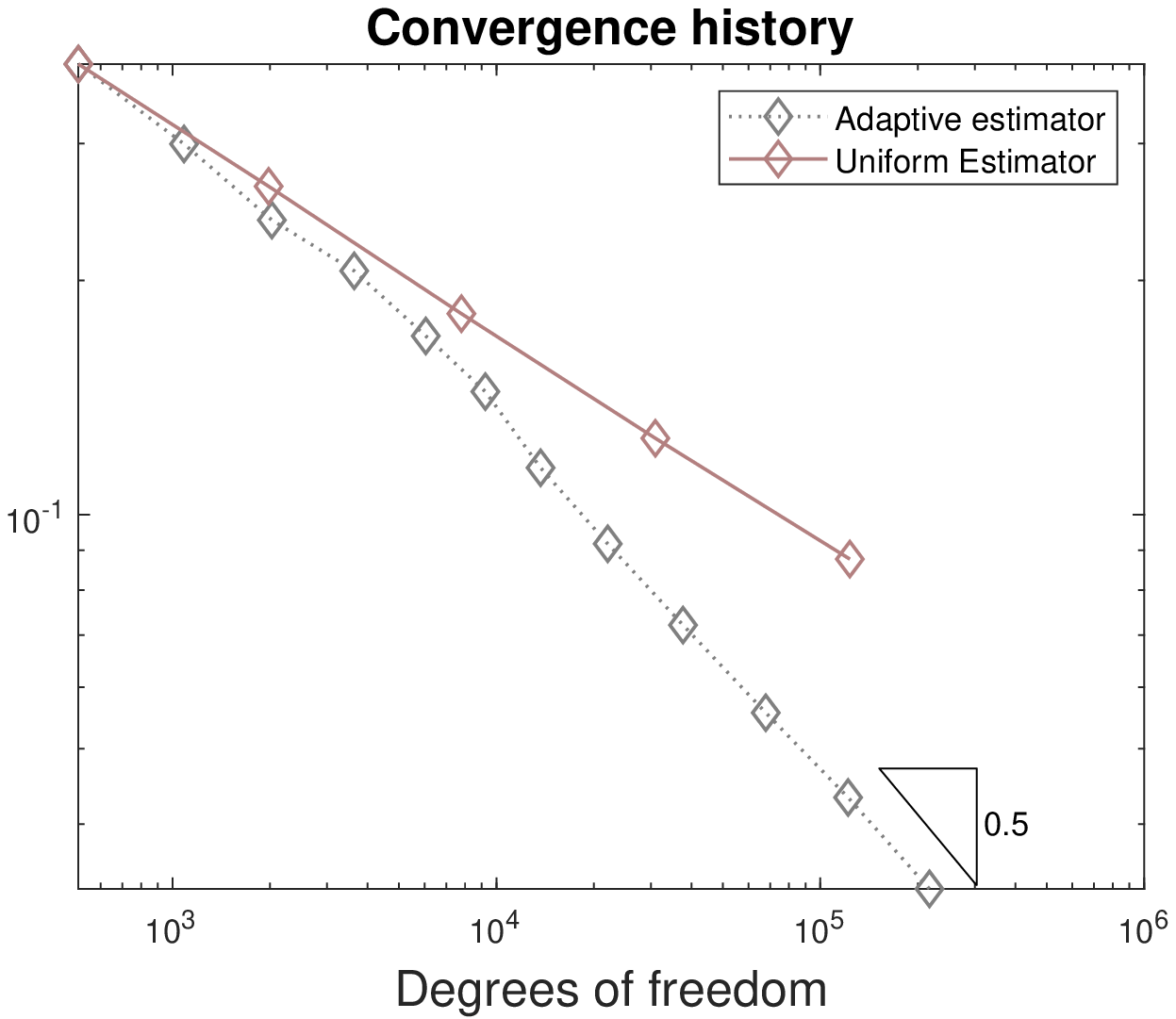}
    \caption{Adaptive mesh patter (left), the corresponding numerical approximation for pressure (middle) and the convergence history for Example~\ref{example3}.}
    \label{ex3}
\end{figure}

\section{Conclusion}

In this paper we developed a residual type error estimator for staggered DG method for Darcy flows in fractured porous media. Our methodology is based on the observation that the error for the conforming counterpart of the discrete solution can be incorporated into the stability of the continuous problem. Thereby we can estimate the error by using some sort of residual, which can be turned into the desired error estimator by combining the information achieved from the discrete formulation. Our approach is different from the one used in \cite{ChenSun17}, where a particular continuous inf-sup condition and $H(\text{div};\Omega)$-conforming interpolation operator are employed. Numerical experiments indicate that our error estimator can work well for multiple non-intersecting fractures. In the future we aim to extend our analysis to the problem with intersecting fractures.

\section*{Acknowledgments}

The research of Eric Chung is partially supported by the Hong Kong RGC General Research Fund (Project numbers 14304719 and 14302018), CUHK Faculty of Science Direct Grant 2019-20 and NSFC/RGC Joint Research Scheme (Project number HKUST620/15).


\begin{thebibliography}{1}


    \bibitem{Alboin02} {\sc C. Alboin, J. Jaffr\'{e}, J. E. Roberts, and C. Serres},
    {\em Modeling fractures as interfaces for flow and transport in porous media}, In Fluid flow and transport in porous media: mathematical and numerical treatment (South Hadley, MA, 2001), volume 295 of Contemp. Math., pages 13–24. Amer. Math. Soc., Providence, RI, 2002.

    \bibitem{Alonso95} {\sc A. Alonso},
    {\em Error estimators for a mixed method}, Numer. Math., 74 (1996), pp. 385--395.

    \bibitem{Antonietti19} {\sc P. F. Antonietti, C. Facciol\`{a}, A. Russo, and M. Verani},
    {\em Discontinuous Galerkin approximation of flows in fractured porous media on polytopic grids},
    SIAM J. Sci. Comput., 41 (2019), pp. A109--A138.

    \bibitem{AntoniettiMF16} {\sc P. F. Antonietti, L. Formaggia, A. Scotti, M. Verani and N. Verzotti},
    {\em Mimetic finite difference approximation of flows in fractured porous media},
    ESAIM Math. Model. Numer. Anal., 50 (2016), pp. 809--832.

    \bibitem{Beir13} {\sc L. Beir\~{a}o da Veiga, F. Brezzi, A. Cangiani, G. Manzini, L. D. Marini, and A. Russo},
    {\em Basic principles of virtual element methods},
    Math. Models Meth. Appl. Sci., 23 (2013), pp. 199--214.

    \bibitem{Beir15} {\sc L. Beir\~{a}o da Veiga and G. Manzini},
    {\em Residual a posteriori error estimation for the virtual rlement method for elliptic problems}, ESAIM Math. Model. Numer. Anal., 49 (2015), pp. 577--599.

    \bibitem{Berrone17} {\sc S. Berrone and A. Borio},
    {\em A residual a posteriori error estimate for the virtual rlement method}, Math. Models Meth. Appl. Sci., 27 (2017), pp. 1423--1458.


    \bibitem{BR78} {\sc I. Babu\v{s}ka and W. C. Rheinboldt},
    {\em Error estimates for adaptive finite element computations}, SIAM J. Numer. Anal., 15 (1978), pp. 736--754.

    \bibitem{Beir08} {\sc L. Beir\~{a}o da Veiga and G. Manzini},
    {\em An a posteriori error estimator for the mimetic finite difference approximation of elliptic problems}, Int. J. Numer. Meth. Engng., 76 (2008), pp. 1696--1723.

    \bibitem{Babuska78} {\sc I. Babu\v{s}ka and W. C. Rheinboldt},
    {\em A posteriori error estimates for the fintie element method}, Int. J. Numer. Methods Engrg., 12 (1978), pp. 1597--1615.

    \bibitem{Benedetto14} {\sc M. F. Benedetto, S. Berrone, S. Pieraccini, and S. Scial\`{o}},
    {\em The virtual element method for discrete fracture network simulations},
    Comput. Meth. Appl. Mech. Eng., 280 (2014), pp. 135--156.

    \bibitem{Bernadi00} {\sc C. Bernardi and R. Verf\"{u}rth},
    {\em Adaptive finite element methods for elliptic equations with non-smooth coefficients}, Numer. Math., 85 (2000), pp. 579--608.

    \bibitem{Bress08} {\sc D. Braess and J. Sch\"{o}berl},
    {\em Equilibrated residual error estimator for edge elements}, Math. Comp., 77 (2008), pp. 651--672.

    \bibitem{Braess96} {\sc D. Braess and R. Verf\"{u}rth},
    {\em A posteriori error estimators for the Raviart-Thomas element}, SIAM J. Numer. Anal., 33 (1996), pp. 2431--2444.

    \bibitem{Carsten97} {\sc C. Carstensen},
    {\em A posteriori error estimates for the mixed finite element method}, Math. Comp., 66 (1997), pp. 465--476.

    \bibitem{ckp11} {\sc C. Carstensen, D. Kim, and E.-J. Park},
    {\em A priori and a posteriori pseudostress-velocity mixed finite element error analysis for the Stokes problem}, SIAM J. Numer. Anal., 49 (2011), pp. 2501-2523.

    \bibitem{Cangiani16} {\sc A. Cangiani, E. H. Georgoulis, T. Pryer, and O. J. Sutton},
    {\em A posteriori error estimates for the virtual element method},
    Numer. Math., 137 (2017), pp. 857--893.

    \bibitem{Chave18} {\sc F. Chave, D. A. Di Pietro, and L. Formaggia},
    {\em A hybrid high-order method for Darcy flows in fractured porous media},
    SIAM J. Sci. Comput., 40 (2018), pp. A1063--A1094.


     \bibitem{chen2016}{\sc H. Chen, A. Salama, and S. Sun},
    {\em Adaptive mixed finite element methods for Darcy flow in fractured porous media}, Water Resour. Res., 52 (2016), pp. 7851--7868.


    \bibitem{ChenSun17} {\sc H. Chen and S. Sun},
    {\em A residual-based a posteriori error estimator for single-phase Darcy flow in fractured porous media},
    Numer. Math., 136 (2017), pp. 805--839.

	\bibitem{Cheung15} {\sc S. W. Cheung, E. Chung, H. H. Kim and Y. Qian},
	{\em Staggered discontinuous Galerkin methods for the incompressible Navier-Stokes equations},
	J. Comput. Phys., 302 (2015), pp. 251--266.



    \bibitem{EricCiarYu13} {\sc E. T. Chung, P. Ciarlet Jr., and T. F. Yu},
    {\em Convergence and superconvergence of staggered discontinuous Galerkin methods for the three-dimensional Maxwell’s equations on Cartesian grids},
    J. Comput. Phy., 235 (2013), pp. 14--31.

    \bibitem{ChungCockburnFu14} {\sc E. T. Chung, B. Cockburn and G. Fu},
    {\em The staggered DG method is the limit of a hybridizable DG method},
     SIAM J. Numer. Anal., 52 (2014), pp. 915--932.

    \bibitem{ChungCockburnFu16} {\sc E. T. Chung, B. Cockburn and G. Fu},
    {\em The staggered DG method is the limit of a hybridizable DG method. Part II: The Stokes flow},
     J. Sci. Comput., 66 (2016), pp. 870--887.

    \bibitem{EricEngquistwave06} {\sc E. T. Chung and B. Engquist},
    {\em Optimal discontinuous Galerkin methods for wave propagation},
    SIAM J. Numer. Anal., 44 (2006), pp. 2131--2158.

    \bibitem{ChungWave} {\sc E. T. Chung and B. Engquist},
    {\em Optimal discontinuous Galerkin methods for the acoustic wave equation in higher dimensions},
    SIAM J. Numer. Anal., 47 (2009), pp. 3820--3848.


    \bibitem{ChungKimWid13} {\sc E. T. Chung, H. H. Kim, and O. B. Widlund},
    {\em Two-level overlapping schwarz algorithms for a staggered discontinuous Galerkin method},
    SIAM J. Numer. Anal., 51 (2013), pp. 47--67.

	\bibitem{ChungLamQian15} {\sc E. T. Chung, C. Y. Lam, and J. Qian},
	{\em A staggered discontinuous Galerkin method for the simulation of seismic waves with surface topography},
	Geophysics, 80 (2015), T119--T135.


    \bibitem{ChungParkLina18} {\sc E. T. Chung, E.-J. Park, and L. Zhao},
    {\em Guaranteed a posteriori error estimates for a staggered discontinuous Galerkin method},
    J. Sci. Comput., 75 (2018), pp. 1079--1101.

    \bibitem{ChungQiu17} {\sc E. T. Chung and W. Qiu},
    {\em Analysis of an SDG method for the incompressible Navier-Stokes equations},
    SIAM J. Numer. Anal., 55 (2017), pp. 543--569.

    \bibitem{DAngelo12} {\sc C. D'Angelo and A. Scotti},
    {\em A mixed finite element method for Darcy flow in fractured porous media with non-matching grids},
    ESAIM Math. Model. Numer. Anal., 46 (2012), pp. 465--489.

    \bibitem{Ern15} {\sc A. Ern and M. Vohral\'{i}k},
    {\em Polynomial-degree-robust a posteriori estimates in a unified setting for conforming, nonconforming, discontinuous Galerkin, and mixed discretizations}, SIAM J. Numer. Anal., 53 (2015), pp. 1058--1081.


    \bibitem{Frih08} {\sc N. Frih, J. E. Roberts, and A. Saada},
    {\em Modeling fractures as interfaces: a model for Forchheimer fractures}, Comput. Geosci., 12 (2008), pp. 91--104.

    \bibitem{Hoteit08} {\sc J. Hoteit and A. Firoozabadi},
    {\em An efficient numerical model for incompressible two-phase flow in fractured media}, Adv. Water Resour., 31 (2008), pp. 891--905.

    \bibitem{Karakashian03} {\sc O. A. Karakashian and F. Pascal},
    {\em A posteriori error estimates for a discontinuous Galerkin approximation of second-order elliptic problems}, SIAM J. Numer. Anal., 41 (2003), pp. 2374--2399.

    \bibitem{kim07} {\sc K. Y. Kim},
    {\em A posteriori error analysis for locally conservative mixed methods}, Math. Compu., 76 (2007), pp. 43--66.

    \bibitem{KimChungLee13} {\sc H. H. Kim, E. T. Chung, and C. S. Lee},
    {\em A staggered discontinuous Galerkin method for the Stokes system},
    SIAM J. Numer. Anal., 51 (2013), pp. 3327--3350.

    \bibitem{kipa-upwind08} {\sc D. Kim and E.-J. Park},
    {\em A posteriori error estimators for the upstream weighting mixed methods for convection diffusion problems}, Comput. Methods in Appl. Mech. and Engrg., 197 (2008), pp. 806--820.

    \bibitem{kimpark-sinum10}  {\sc D. Kim and E.-J. Park},
    {\em A priori and a posteriori analysis of mixed finite element methods for nonlinear elliptic equations}, SIAM J. Numer. Anal., 48 (2010), pp. 1186--1207.

    \bibitem{LarsonAxel08} {\sc M. G. Larson and A. M{\aa}lqvist},
    {\em A posteriori error estimates for mixed finite element approximations of elliptic problems}, Numer. Math., 108 (2008), pp. 487--500.

    \bibitem{LeeKim16} {\sc J. J. Lee and H. H. Kim},
    {\em Analysis of a staggered discontinuous Galerkin method for linear elasticity},
    J. Sci. Comput., 66 (2016), pp. 625--649.

    \bibitem{Martin05} {\sc V. Martin, J. Jaffr\'{e} and J. E. Roberts},
    {\em Modeling fractures and barriers as interfaces for flows in porous media},
    SIAM J. Sci. Comput., 26 (2006), pp. 1667--1691.

    \bibitem{Monteagudo04} {\sc J. Monteagudo and A. Firoozabadi},
    {\em Control-volume method for numerical simulation of two-phase immiscible flow in two- and three-dimensional discrete-fractured media},
    Water Resour. Res., 40 (2004), W07405.



    \bibitem{ScottZhang90} {\sc L. R. Scott and S. Zhang},
    {\em Finite element interpolation of nonsmooth functions satisfying boundary conditions}, Math. Comp., 54 (1990), pp. 483--493.

    \bibitem{Verfurth96} {\sc R. Verf\"{u}rth},
     {\em A review of a posteriori error estimation and adaptive mesh-refinement techniques}, Teubner-Wiley, Stuttgart, 1996.

    \bibitem{martin07} {\sc M. Vohral\'{i}k},
     {\em A posteriori error estimates for lowest-order mixed finite
    element discretizations of convection-diffusion-reaction equations}, SIAM J. Numer. Anal., 45 (2007),
    pp. 1570--1599.

    \bibitem{Vohralik10} {\sc M. Vohral\'{i}k},
     {\em Unified primal formulation-based a priori and a posteriori error analysis of mixed finite element methods}, Math. Comp., 79 (2010), pp. 2001--2032.

    \bibitem{LinaPark} {\sc L. Zhao and E.-J. Park},
    {\em A staggered discontinuous Galerkin method of minimal dimension on quadrilateral and polygonal meshes}, SIAM J. Sci. Comput., 40 (2018), pp. A2543--A2567.

    \bibitem{LinaParkShin} {\sc L. Zhao, E.-J. Park, and D.-w. Shin},
    {\em A staggered DG method of minimal dimension for the Stokes equations on general meshes},
    Comput. Meth. Appl. Mech. Eng., 345 (2019), pp. 854--875.

    \bibitem{LinaParkcoupledSD} {\sc L. Zhao and E.-J. Park},
    {\em A lowest-order staggered DG method for the coupled Stokes-Darcy problem},
    IMA J. Numer. Anal, 2020, doi:10.1093/imanum/drz048.

    \bibitem{LinaEricLam20} {\sc L. Zhao, E. T. Chung and M. F. Lam},
    {\em A new staggered DG method for the Brinkman problem robust in the Darcy and Stokes limits},
    Comput. Meth. Appl. Mech. Eng., 364 (2020).

    \bibitem{LinaDohyun20} {\sc L. Zhao, D. Kim, E.-J. Park, and E. Chung},
    {\em Staggered DG method with small edges for Darcy flows in fractured porous media}, arXiv:2005.10955.

\end{thebibliography}
\end{document}